\documentclass[11pt,letterpaper]{amsart}
\usepackage[utf8]{inputenc}\usepackage[left=2.7cm,right=2.7cm,top=3.5cm,bottom=3cm]{geometry}
\usepackage{amssymb,latexsym,amsmath,amsthm,amscd}
\usepackage{graphicx}
\usepackage{dsfont}
\usepackage{stmaryrd}
\usepackage[pagebackref=false,colorlinks]{hyperref}
\hypersetup{pdffitwindow=true,linkcolor=blue,citecolor=blue,urlcolor=cyan}
\usepackage{hyperref}
\usepackage[all]{xy}
\usepackage{mathrsfs}
\usepackage{color}
\definecolor{Blue}{rgb}{0.3,0.3,0.9}

\usepackage[T2A,OT1]{fontenc}
\DeclareSymbolFont{cyrillic}{T2A}{cmr}{m}{n}
\DeclareMathSymbol{\Sha}{\mathalpha}{cyrillic}{216}

\theoremstyle{plain}
\newtheorem{thm}{Theorem}[subsection] 

\usepackage{mathrsfs}

\theoremstyle{definition}

\newtheorem{defn}[thm]{Definition}

\newtheorem{rem}[thm]{Remark}

\theoremstyle{definition}

\theoremstyle{plain}
\newtheorem{prop}[thm]{Proposition}
\theoremstyle{plain}
\newtheorem{lem}[thm]{Lemma}
\theoremstyle{plain}
\newtheorem{cor}[thm]{Corollary}
\newtheorem*{thm-intro}{Theorem}
\newtheorem*{cor-intro}{Corollary}
\newtheorem*{conj-intro}{Conjecture}

\theoremstyle{plain}
\newtheorem{conjintro}{Conjecture}

\newtheorem{thmintro}{Theorem}

\theoremstyle{plain}

\numberwithin{equation}{section}

\newcommand{\cc}{\mathbf{c}}

\newcommand{\rH}{{\mathrm{H}}}

\newcommand{\lm}{\lambda}

\newcommand{\pp}{\mathfrak{p}}
\newcommand{\ppbar}{\overline{\mathfrak{p}}}

\newcommand{\fa}{\mathfrak{a}}
\newcommand{\bQ}{\mathbf{Q}}
\newcommand{\bZ}{\mathbf{Z}}
\newcommand{\bC}{\mathbf{C}}

\def\cL{{\mathcal L}}

\def\scrO{\mathscr O}

\def\cW{{\mathcal W}}

\def\cO{\mathcal O}

\newcommand{\VQdag}{{\mathbf{V}_{\underline{Q}}^\dagger}}
\newcommand{\Vdag}{{\mathbf{V}^\dagger}}
\newcommand{\Adag}{{\mathbf{A}^\dagger}}

\newcommand{\Vdagsp}{{\mathbb{V}_{\varphi\boldsymbol{g}\boldsymbol{h}}^\dagger}}
\newcommand{\Adagsp}{{\mathbb{A}_{\varphi\boldsymbol{g}\boldsymbol{h}}^\dagger}}

\newcommand{\Vdags}{{\mathbb{V}_{\varphi\underline{\boldsymbol{g}\boldsymbol{h}}}^\dagger}}
\newcommand{\Adags}{{\mathbb{A}_{\varphi\underline{\boldsymbol{g}\boldsymbol{h}}}^\dagger}}

\newcommand{\unb}{{\varphi}}

\newcommand{\Q}{\mathbf{Q}}
\newcommand{\Z}{\mathbf{Z}}

\def\makeop#1{\expandafter\def\csname#1\endcsname
	{\mathop{\rm #1}\nolimits}\ignorespaces}
\makeop{Hom}   \makeop{End}   \makeop{Aut}   
\makeop{Pic} \makeop{Gal}       \makeop{Div} \makeop{Lie}
\makeop{PGL}   \makeop{Corr} \makeop{PSL} \makeop{sgn} \makeop{Spf}
\makeop{Tr} \makeop{Nr} \makeop{Fr} \makeop{disc}
\makeop{Proj} \makeop{supp} \makeop{ker}   \makeop{Im} \makeop{dom}
\makeop{coker} \makeop{Stab} \makeop{SO} \makeop{SL} \makeop{SL}
\makeop{Cl}    \makeop{cond} \makeop{Br} \makeop{inv} \makeop{rank}
\makeop{id}    \makeop{Fil} \makeop{Frac}  \makeop{GL} \makeop{SU}
\makeop{Trd}   \makeop{Sp} \makeop{Tr}    \makeop{Trd} \makeop{Res}
\makeop{ind} \makeop{depth} \makeop{Tr} \makeop{st} \makeop{Ad}
\makeop{Int} \makeop{tr}    \makeop{Sym} \makeop{can} \makeop{SO}
\makeop{torsion} \makeop{GSp} \makeop{Tor}\makeop{Ker} \makeop{rec}
\makeop{Ind} \makeop{Coker}
\makeop{vol} \makeop{Ext} \makeop{gr} \makeop{ad}
\makeop{Gr}\makeop{corank} \makeop{Ann}
\makeop{Hol} 
\makeop{Fitt} \makeop{Mp} \makeop{CAP}

\newcommand{\dBr}[1]{\llbracket{#1}\rrbracket}

\newcommand{\bT}{\mathbb{T}}

\newcommand{\cR}{\mathbb{I}}

\newcommand{\any}{?}

\newcommand{\bfff}{{\boldsymbol{f}}}
\newcommand{\bff}{{\boldsymbol{\varphi}}}
\newcommand{\bfg}{{\boldsymbol{g}}}
\newcommand{\bfh}{{\boldsymbol{h}}}

\begin{document}
	
\title[Generalised Kato classes on 
CM elliptic curves of rank 2]
{Generalised Kato classes on CM elliptic curves of rank 2}
\author[F.~Castella]{Francesc Castella}
	
\date{\today}
		
\address[]{Department of Mathematics, University of California Santa Barbara, CA 93106, USA}
\email{castella@ucsb.edu}




\begin{abstract}
Let $E/\Q$ be a CM elliptic curve and let $p\geq 5$ be a prime of good ordinary reduction for $E$. Suppose that $L(E,s)$ vanishes at $s=1$ and has sign $+1$ in its functional equation, so in particular ${\rm ord}_{s=1}L(E,s)\geq 2$. In this paper we slightly modify a construction of Darmon--Rotger to define a generalised Kato class $\kappa_p\in{\rm Sel}(\Q,V_pE)$, and prove the following rank two analogue of Kolyvagin's result:
\[
\kappa_p\neq 0\quad\Longrightarrow\quad{\rm dim}_{\Q_p}{\rm Sel}(\Q,V_pE)=2.
\]
Conversely, when ${\rm dim}_{\Q_p}{\rm Sel}(\Q,V_pE)=2$ we show that $\kappa_p\neq 0$ \emph{if and only if} the restriction map 
\[
{\rm Sel}(\Q,V_pE)\rightarrow E(\Q_p)\hat\otimes\bQ_p 
\]
is nonzero. The proof of these results, which extend and strenghten similar results of the author with Hsieh in the non-CM case, exploit a new link between the nonvanishing of generalised Kato classes and a main conjecture in anticyclotomic Iwasawa theory. 
\end{abstract}


\maketitle
\setcounter{tocdepth}{1}

\section{Introduction}

Let $E$ be an elliptic curve over the rationals. The systematic construction of  rational points on $E$, akin to the construction of Heegner points when ${\rm ord}_{s=1}L(E,s)\leq 1$ in the works of Gross--Zagier and Kolyvagin from the 1980s, is a well-known open problem lying behind further progress on the Birch--Swinnerton-Dyer conjecture in situations of higher order of vanishing. 
As an approximation to this problem, and following the reformulation and vast generalisation of the conjecture by Bloch--Kato, one might also attempt to construct non-torsion Selmer classes in such situations.

Suppose $L(E,s)$ has sign $+1$ in its functional equation and vanishes at $s=1$ (so ${\rm ord}_{s=1}L(E,s)\geq 2$). Fix a prime $p\geq 5$, and suppose $E$ has CM by an imaginary quadratic field in which $p$ splits. In this paper, we introduce a \emph{generalised Kato class} $\kappa_p\in{\rm Sel}(\Q,V_pE)$ in the $p$-adic Selmer group fitting into the exact sequence
\[
0\rightarrow E(\bQ)\otimes_{\Z_p}\bQ_p\rightarrow{\rm Sel}(\bQ,V_pE)\rightarrow\bQ_p\otimes_{\Z_p}{\rm Ta}_p\Sha(E/\bQ)\rightarrow 0,
\]
where ${\rm Ta}_p\Sha(E/\bQ)$ is the $p$-adic Tate module for the Shafarevich--Tate group of $E$. The class $\kappa_p$ is obtained from a  ``twisted variant'' of a construction due to Darmon--Rotger \cite{DR2.5}, whence the terminology. One of the main results in this paper is a proof of the following rank two analogue of Kolyvagin's result \cite{kol88}:
\[
\kappa_p\neq 0\quad\Longrightarrow\quad{\rm dim}_{\Q_p}{\rm Sel}(\Q,V_pE)=2.
\]
Conversely, still in the setting that $L(E,s)$ vanishes to positive even order at $s=1$, we show that if ${\rm dim}_{\Q_p}{\rm Sel}(\Q,V_pE)=2$ then $\kappa_p\neq 0$ \emph{if and only if} the restriction map
\[
{\rm Sel}(\bQ,V_pE)\rightarrow E(\Q_p)\hat{\otimes}\Q_p
\]
is nonzero, where $E(\Q_p)\hat{\otimes}\Q_p$ is the $p$-adic completion $\varprojlim_nE(\bQ_p)/p^nE(\bQ_p)$ tensored with $\bQ_p$. 

In the non-CM case, similar results (for the generalised Kato classes $\kappa_{p}^{\rm DR}$ introduced in \cite{DR2.5}) were obtained in a joint work of the author with M.-L.\,Hsieh \cite{cas-hsieh-ord}. The approach introduced in this paper to handle the CM case (a case that is essentially excluded by the approach in \emph{loc.\,cit.}; for instance, it requires $E$ to have some prime $\ell\neq p$ of multiplicative reduction) also yields a new proof of the original results in the non-CM case (see \cite{cas-GKC}).

In the rest of this Introduction we explain our results more precisely, and  some key ideas behind the proof.


\subsection{Diagonal cycle main conjecture}

Following the work of Darmon--Rotger \cite{DR2}, 
as further systematically developed also by Bertolini--Seveso--Venerucci (see \cite{DR3,BSV}), attached to a triple of Hida families $(\bff,\bfg,\bfh)$ with tame characters having product $\chi_\varphi\chi_g\chi_h=\omega^{2a}$ for some $a\in\Z$, where $\omega:(\Z/p\Z)^\times\rightarrow\Z_p^\times$ is the Teichm\"uller character, one has a \emph{big diagonal class}
\[
\kappa(\bff,\bfg,\bfh)\in\rH^1(\bQ,\mathbf{V}^\dagger),
\]
where $\mathbf{V}^\dagger$ is a self-dual twist of the triple tensor product of the Galois representations associated to $\bff,\bfg,\bfh$. It follows from its geometric construction (interpolating $p$-adic \'{e}tale Abel--Jacobi images of generalised Gross--Kudla--Schoen diagonal cycles in $p$-adic families) that $\kappa(\bff,\bfg,\bfh)$ lands in the \emph{balanced} Selmer group ${\rm Sel}^{\rm bal}(\bQ,\mathbf{V}^\dagger)$.

Denote by $\mathcal{R}=\cR_\varphi\hat\otimes_{\mathscr{O}}\cR_g\hat\otimes_{\mathscr{O}}\cR_h$ the completed tensor product of the ring of definition of $\bff,\bfg,\bfh$; this is a finite extension of the three-variable Iwasawa algebra $\Lambda\hat\otimes_{\mathscr{O}}\Lambda\hat\otimes_{\mathscr{O}}\Lambda$, where $\Lambda=\mathscr{O}\dBr{1+p\bZ_p}$ and $\mathscr{O}$ is a finite extension of $\bZ_p$. Let $\underline{Q}=(Q_0,Q_1,Q_2)\in{\rm Spec}(\mathcal{R})(\overline{\bQ}_p)$ run over the arithmetic specialisations of $\mathcal{R}$, as defined in $\S\ref{subsec:triple}$. Under the root number condition that for some $\underline{Q}$ we have
\begin{equation}\label{eq:H}
\varepsilon_\ell(\mathbf{V}_{\underline{Q}}^\dagger)=+1\quad\textrm{for all primes}\quad\ell\mid N_\varphi N_g N_h \tag{H}
\end{equation}
(a condition that is known to be independent of $\underline{Q}$), 
the sign $\varepsilon(\mathbf{V}_{\underline{Q}}^\dagger)\in\{\pm{1}\}$ in the functional equation of the triple product $L$-function $L(\mathbf{V}_{\underline{Q}}^\dagger,s)$ depends only on the  local root number of $\mathbf{V}_{\underline{Q}}^\dagger$ at $\ell=\infty$, which in turn depends only on the relative position of the weights $k_{Q_0},k_{Q_1},k_{Q_2}$ of the corresponding specialisations of $\bff,\bfg,\bfh$. In particular, one finds
\[
\varepsilon(\mathbf{V}_{\underline{Q}}^\dagger)=\varepsilon_\infty(\mathbf{V}_{\underline{Q}}^\dagger)=\begin{cases}
-1&\textrm{if $k_{Q_0}+k_{Q_1}+k_{Q_2}>2k_{Q_i}$ for all $i=0,1,2$,}\\[0.3em]
+1&\textrm{if $k_{Q_0}\geq k_{Q_1}+k_{Q_2}$.}
\end{cases}
\]
In the first (resp. second) case, we say that $\underline{Q}$ is in the \emph{balanced} (resp. $\varphi$-unbalanced) range. As shown in \cite{DR3,BSV}, the specialisations of $\kappa(\bff,\bfg,\bfh)$ in the balanced range recover generalised Gross--Kudla--Schoen diagonal cycles; the interplay between this and the $p$-adic interpolation of the central $L$-values $L(\mathbf{V}_{\underline{Q}}^\dagger,0)$ in the $\varphi$-unbalanced region plays a key role in this paper.
%

Motivated by Perrin-Riou's Heegner point main conjecture \cite{PR-HP} (and more precisely, Howard's extension in the context of big Heegner points \cite{howard-invmath}), the following is expected about $\kappa(\bff,\bfg,\bfh)$:

\begin{conjintro}[Big diagonal class main conjecture]\label{conj:DCMC}
Assume hypothesis \eqref{eq:H}. 
Then $\kappa(\bff,\bfg,\bfh)$ is not $\mathcal{R}$-torsion, the modules ${\rm Sel}^{\rm bal}(\Q,\mathbf{V}^\dagger)$ and $X^{\rm bal}(\bQ,\mathbf{A}^\dagger)$ both have $\mathcal{R}$-rank one, and 
\[
{\rm char}_{\mathcal{R}}\bigl(X^{\rm bal}(\Q,\mathbf{A}^\dagger)_{\rm tors}\bigr)={\rm char}_{\mathcal{R}}\biggl(\frac{{\rm Sel}^{\rm bal}(\Q,\mathbf{V}^\dagger)}{\mathcal{R}\cdot\kappa(\bff,\bfg,\bfh)}\biggr)^2
\]
in $\mathcal{R}\otimes\Q_p$, where the subscript ${\rm tors}$ denotes the $\mathcal{R}$-torsion submodule.
\end{conjintro}

Here $X^{\rm bal}(\bQ,\mathbf{A}^\dagger)={\rm Hom}({\rm Sel}(\bQ,\mathbf{A}^\dagger),\bQ_p/\bZ_p)$ is the Pontryagin dual of the balanced Selmer group 
with coefficients in $\mathbf{A}^\dagger={\rm Hom}_{\Z_p}(\mathbf{V}^\dagger,\mu_{p^\infty})$ (see $\S\ref{subsec:Sel-triple}$).

When the Hida families $\bfg,\bfh$ are specialised to classical modular forms $g,h$ of weights $l,m\geq 2$ with $l\equiv m\pmod{2}$ and $\bff$ is a CM Hida family, denoting by $\mathbb{V}^\dagger_{\bff gh}$ the resulting specialisation of $\mathbf{V}^\dagger$, the divisibility
\[
{\rm char}_{\mathcal{R}_{\bff gh}}\bigl(X^{\rm bal}(\Q,\mathbb{A}_{\bff gh}^\dagger)_{\rm tors}\bigr)\supset{\rm char}_{\Lambda_{\rm ac}}\biggl(\frac{{\rm Sel}^{\rm bal}(\Q,\mathbb{V}_{\bff gh}^\dagger)}{\mathcal{R}_{\bff gh}\cdot\kappa(\bff,g,h)}\biggr)^2
\]
in $\mathcal{R}_{\bff gh}\otimes\Q_p$  
was proved under some hypotheses (including the non-triviality of $\kappa(\bff,g,h)$) in \cite[Thm.~9.10]{ACR} 
by constructing an anticyclotomic Euler system (in the sense of Jetchev--Nekov{\'a}{\v{r}}--Skinner \cite{JNS}) having the specialised big diagonal class $\kappa(\bff,g,h)$ as its bottom class. On the other hand, closer to the setting of this paper, when $\bff$ is specialised to a classical modular form $\varphi$ of weight $2r\geq 2$ and $\bfg,\bfh$ are both CM Hida families with respect to the same imaginary quadratic field, a similar divisibility is obtained in \cite{C-Do} 
by building an anticyclotomic Euler system containing the resulting specialisation of $\kappa(\bff,\bfg,\bfh)$ as its bottom class. However, both of these results are subject to a ``big image'' hypothesis (for $(g,h)$ and for $\varphi$, respectively) which excludes the CM case.


The first main result of this paper is the proof of a two-variable specialisation of Conjecture~\ref{conj:DCMC} in the CM case. We consider the case in which $\varphi=\theta(\lambda_0)\in S_2(\Gamma_1(N_\varphi))$ has CM by an imaginary quadratic field $K$ in which $p=\pp\ppbar$ splits, and 
\[
\bfg=\boldsymbol{\theta}_{\lambda_1}(S_1),\quad\bfh=\boldsymbol{\theta}_{\lambda_2}(S_2)
\] 
are CM Hida families 
by the same $K$. 
Here $\lambda_0$ (resp. $\lambda_1,\lambda_2$) is a Hecke character of $K$ of infinity type $(-1,0)$ (resp. finite order) subject to the self-duality condition
\begin{equation}\label{eq:sd-lambda-intro}
\chi_{\lambda_0}\chi_{\lambda_1}\chi_{\lambda_2}=\eta_{K/\Q},
\end{equation}
where $\chi_{\lambda_i}$ is the central character of $\lambda_i$, and $\eta_{K/\Q}$ denotes the quadratic character corresponding to $K/\Q$. We assume  the conductor $N_\varphi$ of $\varphi$ is coprime to $p$, and let $N_g, N_h$ denote the tame conductor of $\bfg,\bfh$, respectively. Since it will suffice for our application, we also assume for simplicity that the class number of $K$ is coprime to $p$.

Letting $\mathbb{V}_{\varphi\bfg\bfh}^\dagger$ denote the resulting two-variable specialisation of $\mathbf{V}^\dagger$, under mild hypotheses on $\bfg$ and $\bfh$ (see Proposition~\ref{prop:CM-ind}) it is easy to see that
\begin{equation}\label{eq:dec-intro}
\begin{aligned}
\mathbb{V}_{\varphi\bfg\bfh}^\dagger&\simeq{\rm Ind}_{K}^\Q\bigl(\lambda_0^{-1}\lambda_1^{-1}\lambda_2^{-1}\Psi_{W_1}^{1-\cc}\bigr)\oplus{\rm Ind}_{K}^\Q\bigl(\lambda_0^{-1}\lambda_1^{-\cc}\lambda_2^{-\cc}\Psi_{W_1}^{\cc-1}\bigr)\\
&\quad\oplus{\rm Ind}_{K}^\Q\bigl(\lambda_0^{-1}\lambda_1^{-1}\lambda_2^{-\cc}\Psi_{W_2}^{1-\cc}\bigr)\oplus{\rm Ind}_{K}^\Q\bigl(\lambda_0^{-1}\lambda_1^{-\cc}\lambda_2^{-1}\Psi_{W_2}^{\cc-1}\bigr),
\end{aligned}
\end{equation}
where $W_1=\mathbf{u}^{-1}(1+S_1)^{1/2}(1+S_2)^{1/2}-1$ and
$W_2=(1+S_1)^{1/2}(1+S_2)^{-1/2}-1$ are formal variables parametrising anticyclotomic weight space, with $\mathbf{u}=1+p$.  By \eqref{eq:sd-lambda-intro}, the Hecke characters
\begin{equation}\label{eq:lambda-intro}
\lambda_0^{-1}\lambda_1^{-1}\lambda_2^{-1},\quad \lambda_0^{-1}\lambda_1^{-\cc}\lambda_2^{-\cc},\quad 
\lambda_0^{-1}\lambda_1^{-\cc}\lambda_2^{-\cc},\quad 
\lambda_0^{-1}\lambda_1^{-\cc}\lambda_2^{-1}
\end{equation}
appearing in \eqref{eq:dec-intro} are all self-dual, in the sense that their associated Hecke $L$-function is self-dual, with a functional equation relating its values at $s$ and $-s$. Let ${\rm sign}(\lambda_0\lambda_1\lambda_2)\in\{\pm{1}\}$ be the sign in the functional equation for the Hecke $L$-function  $L(\lambda_0^{-1}\lambda_1^{-1}\lambda_2^{-1},s)$, and similarly for the other three characters in \eqref{eq:lambda-intro}. 

The ordinary $p$-stabilisation of $\varphi$ can be obtained as a weight $2$ specialisation of a unique CM Hida family $\boldsymbol{\varphi}$, and we let $\mathcal{R}_{\varphi\bfg\bfh}\simeq\mathscr{O}\dBr{S_1,S_2}$ be the resulting specialisation of the coefficient ring $\mathcal{R}$. We can now state the first main result of this paper.

\begin{thmintro}\label{thmintro:A}
Let the triple
\[
(\varphi,\bfg,\bfh)=(\theta(\lambda_0),\boldsymbol{\theta}_{\lambda_1}(S_1),\boldsymbol{\theta}_{\lambda_2}(S_2)) 
\]
be as above,
and suppose that:
\begin{itemize}
\item[(i)] 
$\lambda_i(\pp)\not\equiv\lambda_i(\ppbar)\pmod{p}$ for $i=0,1,2$.
\item[(ii)] ${\rm sign}(\lambda)={\rm sign}(\lambda(\lambda_1\lm_2)^{\cc-1})={\rm sign}(\lambda\lambda_2^{\cc-1})={\rm sign}(\lambda\lambda_1^{\cc-1})=+1$. 
\end{itemize}
Then $\kappa(\varphi,\bfg,\bfh)$ is not $\mathcal{R}_{\varphi\bfg\bfh}$-torsion, the modules ${\rm Sel}^{\rm bal}(\Q,\mathbb{V}_{\varphi\bfg\bfh}^\dagger)$ and $X^{\rm bal}(\bQ,\mathbb{A}_{\varphi\bfg\bfh}^\dagger)$ both  have $\mathcal{R}_{\varphi\bfg\bfh}$-rank one, and 
\[
{\rm char}_{\mathcal{R}_{\varphi\bfg\bfh}}\bigl(X^{\rm bal}(\Q,\mathbb{A}_{\varphi\bfg\bfh}^\dagger)_{\rm tors}\bigr)={\rm char}_{\mathcal{R}_{\varphi\bfg\bfh}}\biggl(\frac{{\rm Sel}^{\rm bal}(\Q,\mathbb{V}_{\varphi\bfg\bfh}^\dagger)}{\mathcal{R}_{\varphi\bfg\bfh}\cdot\kappa(\varphi,\bfg,\bfh)}\biggr)^2
\]
in $\mathcal{R}_{\varphi\bfg\bfh}\otimes\Q_p$. In other words, Conjecture~\ref{conj:DCMC} holds for $(\varphi,\bfg,\bfh)$.
\end{thmintro}

A key input in the proof of this result is the relation between Conjecture~\ref{conj:DCMC} and the anticyclotomic main conjecture for Hecke characters. Indeed, in the setting of Theorem~\ref{thmintro:A}, we show that the balanced Selmer group ${\rm Sel}^{\rm bal}(\bQ,\mathbb{V}_{\varphi\bfg\bfh}^\dagger)$ decomposes as
\begin{equation}\label{eq:dec-bal-intro}
\begin{aligned}
{\rm Sel}^{\rm bal}(\bQ,\mathbb{V}_{\varphi\bfg\bfh}^\dagger)&\simeq
{\rm Sel}_{\emptyset,0}(K,T_{\lambda}\otimes\Psi_{W_1}^{1-\cc})\oplus{\rm Sel}_{0,\emptyset}(K,T_{\lambda(\lambda_1\lambda_2)^{\cc-1}}\otimes\Psi_{W_1}^{\cc-1})\\
&\quad\oplus{\rm Sel}_{\emptyset,0}(K,T_{\lambda\lambda_2^{\cc-1}}\otimes\Psi_{W_2}^{1-\cc})\oplus{\rm Sel}_{\emptyset,0}(K,T_{\lambda\lambda_1^{\cc-1}}\otimes\Psi_{W_2}^{\cc-1}).
\end{aligned}
\end{equation}
The Selmer groups in the right-hand side of this decomposition correspond to usual (i.e., Bloch--Kato) anticyclotomic Selmer groups attached to Hecke characters, except for the Selmer group
\[
{\rm Sel}_{0,\emptyset}(K,T_{\lambda(\lambda_1\lm_2)^{\cc-1}}\otimes\Psi_{W_1}^{\cc-1}),
\]
which is obtained from the usual anticyclotomic Selmer group for $T_{\lambda(\lambda_1\lambda_2)^{\cc-1}}$ by \emph{reversing} the local conditions at the primes above $\pp$ and $\ppbar$. 
%
We exploit the fact that Conjecture~\ref{conj:DCMC} follows from 
the Iwasawa--Greenberg main conjecture for the 
$p$-adic triple product $L$-function $\mathscr{L}_p^\varphi(\varphi,\bfg,\bfh)\in\mathcal{R}_{\varphi\bfg\bfh}$ 
constructed by Hsieh \cite{hsieh-triple}. Under certain conditions  preventing the vanishing of $\mathscr{L}_p^\varphi(\varphi,\bfg,\bfh)$ for sign reasons, this main conjecture predicts that the \emph{$\varphi$-unbalanced} Selmer group $X^\varphi(\Q,\mathbb{A}_{\varphi\bfg\bfh}^\dagger)$ is $\mathcal{R}_{\varphi\bfg\bfh}$-torsion, with characteristic ideal generated by $\mathscr{L}_p^\varphi(\varphi,\bfg,\bfh)^2$. We prove a decomposition for $X^\varphi(\bQ,\mathbb{A}_{\varphi\bfg\bfh}^\dagger)$ analogous to \eqref{eq:dec-bal-intro} in which (contrary to the case of the  balanced Selmer groups) \emph{all} direct summands agree with classical anticyclotomic Selmer groups for Hecke characters. Together with a parallel factorisation for $\mathscr{L}_p^\varphi(\varphi,\bfg,\bfh)^2$ into a product of four anticyclotomic Katz $p$-adic $L$-functions, we thus deduce from the works of Agboola--Howard \cite{AHsplit} and Arnold \cite{Arnold} (an anticyclotomic specialisation of Rubin's proof of the Iwasawa main conjecture for $K$ \cite{rubin-IMC}) a proof of the Iwasawa--Greenberg main conjecture for $X^\varphi(\bQ,\mathbb{A}^\dagger_{\varphi\bfg\bfh})$. (Condition (ii) of Theorem~\ref{thmintro:A} 
is needed at this point,  
as otherwise our results 
imply that $X^\varphi(\bQ,\mathbb{A}_{\varphi\bfg\bfh}^\dagger)$ has positive $\mathcal{R}_{\varphi\bfg\bfh}$-rank.) The proof of Theorem~\ref{thmintro:A} then follows. 

\subsection{Generalised Kato classes attached to $E$}

We keep $K$ to be an imaginary quadratic field in which $p=\pp\ppbar$ splits, and now let $E/\bQ$ be an elliptic curve with CM by the ring of integers of $K$. Note that the splitting condition on $p$ implies that $E$ has good ordinary reduction of $p$. Let $\psi_E$ be the Hecke character of $K$ attached to $E$, so that 
\[
L(E,s)=L(\psi_E,s).
\]
Suppose  the sign is the functional equation of $L(E,s)$ is $w=+1$.  
In the second part of the paper, we choose Hecke characters $\lambda_0,\lambda_1,\lambda_2$ for $K$ as above satisfying  
\begin{equation}\label{eq:nonvan-L}
\psi_E=\lambda_0\lambda_1^\cc\lambda_2^\cc,\quad
L(\psi_E^{-1}\lambda_1^{\cc-1},0)\cdot L(\psi_E^{-1}\lambda_2^{\cc-1},0)\cdot L(\psi_E^{-1}(\lambda_1\lambda_2)^{\cc-1},0)\neq 0,
\end{equation}
whose existence follows from nonvanishing results due to Greenberg and Rohrlich \cite{greenberg-critical,rohrlich-ac}, and --- inspired by a construction of generalised Kato classes due to Darmon--Rotger \cite{DR2,DR2.5} ---  we let
\[
\kappa_p\in\rH^1(\Q,V_pE)
\]
be the image of the resulting big diagonal class $\kappa(\varphi,\bfg,\bfh)$  (for $(\varphi,\bfg,\bfh)=(\theta(\lambda_0),\boldsymbol{\theta}_{\lambda_1}(S_1),\boldsymbol{\theta}_{\lambda_2}(S_2))$ as above)  
under the composition
\[
\rH^1(\bQ,\mathbb{V}_{\varphi\bfg\bfh}^\dagger)\rightarrow\rH^1\bigl(\bQ,{\rm Ind}_K^{\bQ}(T_{\lambda_0\lambda_1^{\cc}\lambda_2^{\cc}}\otimes\Psi_{W_1}^{\cc-1})\bigr)\rightarrow\rH^1(\bQ,{\rm Ind}_K^{\Q}(T_{\psi_E}))\simeq\rH^1(\Q,T_pE)
\]
arising from projection onto the second direct summand in \eqref{eq:dec-intro} and the specialisation at $W_1=0$. The construction of $\kappa_p$ might be seen as a twisted variant of the construction of geneneralised Kato classes in \cite{DR2.5}, which for an elliptic curve $E/\bQ$ as above would take $\lambda_0=\psi_E$ and the finite order Hecke characters $\lambda_1,\lambda_2$  to be inverses of each other (similarly as in \cite{cas-hsieh-ord}).
Nevertheless, from the explicit reciprocity law of \cite{DR3} and \cite{BSV} we deduce the implication
\[
L(E,1)=0\quad\Longrightarrow\quad\kappa_p\in{\rm Sel}(\bQ,V_pE).
\]
Since we assume that $L(E,s)$ has sign $+1$, the vanishing of $L(E,1)$ implies that ${\rm ord}_{s=1}L(E,s)\geq 2$, and so by the Bloch--Kato conjecture \cite{BK}, the Selmer group ${\rm Sel}(\Q,V_pE)$ is expected to be at least $2$-dimensional. Our next result is consistent with this expectation, and further justifies the view of $\kappa_p$ as a ``rank $2$ $p$-adic regulator''.

\begin{thmintro}
\label{thmintro:B}
%
Suppose $L(E,s)$ vanishes to positive even order at $s=1$. 
%
Let $\lambda_0,\lambda_1,\lambda_2$ be any triple of Hecke characters of $K$ as above satisfying \eqref{eq:nonvan-L} and the conditions (i)--(ii) in Theorem~\ref{thmintro:A}, and let $\kappa_p$ be the associated generalised Kato class.  Then 
\[
\kappa_p\neq 0\quad\Longrightarrow\quad{\rm dim}_{\Q_p}{\rm Sel}(\Q,V_pE)=2.
\]
Conversely, if ${\rm dim}_{\Q_p}{\rm Sel}(\Q,V_pE)=2$ then $\kappa_p\neq 0$ if and only if the restriction map
\[
{\rm res}_p:{\rm Sel}(\bQ,V_pE)\rightarrow E(\bQ_p)\hat\otimes\bQ_p
\]
is nonzero. 
\end{thmintro}

The existence of (infinitely many) triples $\lambda_0,\lambda_1,\lambda_2$ satisfying \eqref{eq:nonvan-L} and the conditions (i)-(ii) in Theorem~\ref{thmintro:A} follows easily from the aforementioned nonvanishing results due to Greenberg and Rorhlich (note that condition \eqref{eq:sd-lambda-intro} is implied by $\psi_E=\lambda_0\lambda_1^\cc\lambda_2^\cc$, and condition (ii) is implied by  \eqref{eq:nonvan-L} and our root number assumption on $E$). 

In addition to Theorem~\ref{thmintro:A}, the ingredients in the proof of Theorem~\ref{thmintro:B} are a version of Mazur's control theorem for the Selmer groups in the decomposition \eqref{eq:dec-intro} and a global duality argument allowing us to relate the rank of ${\rm Sel}_{0,\emptyset}(K,T_{\psi_E})$ and ${\rm dim}_{\bQ_p}{\rm Sel}(\bQ,V_pE)$.

The result of Theorem~\ref{thmintro:B} is consistent with predictions by Darmon--Rotger \cite{DR2.5}, and  it offers some new insights. More precisely, that the conditions ${\rm dim}_{\Q_p}{\rm Sel}(\bQ,V_pE)=2$ and ${\rm res}_p\neq 0$ imply $\kappa_p\neq 0$ is suggested by [\emph{op.\,cit.}, Conj.~3.12] (in the ``rank $(2,0)$ setting'' of $\S{4.5.3}$); a new insight of Theorem~\ref{thmintro:B} is that, when ${\rm Sel}(\Q,V_pE)$ is $2$-dimensional, the condition ${\rm res}_p\neq 0$ is also \emph{necessary} for the nonvanishing of $\kappa_p$.

\subsection{Application to rank two Selmer basis}

The construction of $\kappa_p\in{\rm Sel}(\bQ,V_pE)$ depends on a choice of $\lambda_0,\lambda_1,\lambda_2$, but it follows from our results 
that different choices give rise to the same Selmer class up to scaling. (In fact, whenever nonzero, $\kappa_p$ generates the one-dimensional subspace ${\rm ker}({\rm res}_p)\subset{\rm Sel}(\bQ,V_pE)$.) 
It is then natural to ask for a class in the two-dimensional ${\rm Sel}(\Q,V_pE)$ complementary to the line spanned by $\kappa_p$. Our results also yield an answer to this question under some hypotheses.

%

Let $\cW$ be the completion of the ring of integers of the maximal unramified extension of $\bQ_p$, and let $\cL_{\pp}=\cL_{\pp,\mathfrak{f}}\in\cW\dBr{Z(\mathfrak{f})}$ be the Katz $p$-adic $L$-function recalled in Theorem~\ref{thm:katz} below, where $Z(\mathfrak{f})$ is the Galois group of the extension $K(E[p^\infty])/K$ and $\mathfrak{f}\subset\cO_K$ is the conductor of $\psi_E$.  For $s\in\Z_p$ define
\[
L_\pp(s)=\cL_{\pp}(\psi_E\langle\psi_E\rangle^{s-1}),\quad
L_\pp^*(s)=\cL_\pp(\psi_E^\cc\langle\psi_E^\cc\rangle^{s-1}),
\]
where $\langle-\rangle:\Z_p^\times\rightarrow 1+p\Z_p$ is the projection onto the $1$-units. 


\begin{thmintro}\label{thm:2}
Let the hypotheses be as in Theorem~\ref{thmintro:B}, and assume in addition that $\Sha(E/\Q)[p^\infty]$ is finite and 
the following conditions hold:
\begin{equation}\label{eq:ord=2}
{\rm ord}_{s=1}L_{\pp}(s)=2,\quad{\rm ord}_{s=1}L_{\pp}^*(s)=1.
\end{equation}
Then ${\rm Sel}(\Q,V_pE)$ is $2$-dimensional, with
\[
{\rm Sel}(\Q,V_pE)=\Q_p\kappa_p\oplus\Q_p x_\pp^{(2)},
\]
where $\kappa_p$ is a generalised Kato class and $x_\pp^{(2)}$ is a ``derived'' elliptic unit.
\end{thmintro}

\begin{proof}
The assumption that ${\rm ord}_{s=1}L_\pp(s)=2$ implies that 
\[
r:={\rm dim}_{\Q_p}{\rm Sel}(\Q,V_pE)\leq 2
\] 
by Rubin's proof of the Iwasawa main conjecture for $K$ and the work of Perrin-Riou (see \cite[Thm.\,4.1]{rubin-IMC} and \cite[Ch.~IV, Thm.\,22]{PR-BSMF}). On the other hand, it also implies $r\geq 2$ by 
the theorem of Coates--Wiles \cite{CW} and the $p$-parity conjecture \cite{guo-parity}. Therefore $r=2$, and by \cite[Prop.~4.4]{rubin-points} the construction of derived elliptic units in  [\emph{op.\,cit.}, \S{6}] yields a class
\[
x_\pp^{(2)}\in{\rm Sel}(K,T_{\pp}E)\simeq{\rm Sel}(\Q,T_pE)
\]
where $T_\pp E$ is the $\pp$-adic Tate module of $E$.
Since by \cite[Thm.\,9.5(ii)]{rubin-points} 
and our assumptions we have ${\rm res}_p(x_\pp^{(2)})\neq 0$, the result follows from Theorem~\ref{thmintro:B}. 
\end{proof}

\begin{rem}
The rank $2$ case of the $\pp$-adic Birch--Swinnerton-Dyer conjecture \cite[Conj.~B]{BGS-pBSD} (resp. Rubin's variant \cite[\S{1.5}]{rubin-BSD}) predicts the equivalence 
\[
{\rm ord}_{s=1}L_\pp(s)=2\;\textrm{(resp. ${\rm ord}_{s=1}L_\pp^*(s)=1$)}\quad\overset{?}\Longleftrightarrow\quad{\rm dim}_{\Q_p}{\rm Sel}(\Q,V_pE)=2.
\]
As shown in the proof of Theorem~\ref{thm:2}, the implication ``$\Longrightarrow$'' follows from known results, and the new contribution here is the \emph{explicit construction} of a $\Q_p$-basis for ${\rm Sel}(\Q,V_pE)$. Conversely, by \cite[Ch.~IV, Thm.~22]{PR-BSMF} (resp. \cite[Thm.~5]{rubin-BSD}) we know that
\[
{\rm dim}_{\Q_p}{\rm Sel}(\Q,V_pE)=2\quad\Longrightarrow\quad{\rm ord}_{s=1}L_\pp(s)\geq 2\;\textrm{(resp. ${\rm ord}_{s=1}L_\pp^*(s)\geq 1$)},
\]
with equality if and only if the $\pp$-adic height pairing $\langle\,,\,\rangle_{\pp}$ of \cite[Ch.~IV]{PR-BSMF} associated to $\psi_E$ is non-degenerate (resp. $\langle\,,\,\rangle_{\pp}$ is non-degenerate, $\#\Sha(E/\Q)[\pp^\infty]<\infty$ and $\langle\kappa^*_p,\kappa^*_p\rangle_{\pp}\neq 0$, where 
\[
\kappa^*_p={\rm log}_{E,p}(Q)P-{\rm log}_{E,p}(P)Q
\]
for a $\Q_p$-basis $(P,Q)$ of ${\rm Sel}(\Q,V_pE)$, and with ${\rm log}_{E,p}:E(\bQ_p)\rightarrow\Q_p$ the formal group logarithm.) It is interesting to note that our proof of Theorem~\ref{thmintro:B} shows that $\kappa_p\equiv\kappa_p^*\pmod{\Q_p^\times}$.
\end{rem}

\subsection{Acknowledgements} 

It is a pleasure to thank Ashay Burungale for stimulating exchanges, especially about the idea to consider a setting along the lines of that in $\S\ref{subsec:setting}$. We are also grateful to Ming-Lun Hsieh for his comments on an early draft.  
Finally, we thank the referees for a number of helpful comments and suggestions that led to significant improvements in the exposition. 

At different stages during the preparation of this paper, the author was supported by the NSF grants DMS-2101458 and DMS-2401321, and the 2024-2025 AMS Centennial Research Fellowship.

\section{$p$-adic $L$-functions}

In this section we recall the two $p$-adic $L$-functions that will appear in our arguments, one due to Katz \cite{Katz49} attached to Hecke characters of an imaginary quadratic field, and another due to Hsieh \cite{hsieh-triple} (extending and refining earlier constructions due to Harris--Tilouine \cite{harris-tilouine} and Darmon--Rotger \cite{DR1})  attached to triple products of modular forms in Hida families. 

Fix a prime $p>2$ and an imaginary quadratic field $K$ with ring of integers $\cO_K$ in which
\begin{equation}\label{eq:spl}
\textrm{$(p)=\pp\ppbar$ splits,}\tag{spl}
\end{equation}
with $\pp$ the prime of $K$ above $p$ determined by a fixed embedding $\iota_p:\overline{\Q}\hookrightarrow\overline{\Q}_p$.

\subsection{Katz $p$-adic $L$-function}

Denote by $D_K<0$ the discriminant of $K$, and fix an integral ideal $\mathfrak{C}\subset\cO_K$ coprime to $p$. Let $\mathcal{W}$ be 
the Witt ring $W(\overline{\mathbf{F}}_p)$ (or a finite extension thereof, later in the paper), and denote by $Z(\mathfrak{C})$ the Galois group of the ray class field extension $K(\mathfrak{C}p^\infty)/K$. 

We say that a Hecke character $\chi:K^\times\backslash\mathbb{A}_K^\times\rightarrow\bC^\times$ has infinity type $(a,b)$ if $\chi_\infty(z)=z^a\overline{z}^b$. We take \emph{geometric} normalisations for the reciprocity maps of class field theory. Let $\mathfrak{f}_\chi$ denote the conductor of $\chi$, and suppose the prime-to-$p$ part $\mathfrak{f}_\chi^{(p)}$ of $\mathfrak{f}_\chi$ divides $\mathfrak{C}$. Viewing $\chi$ as a $\overline{\Q}$-valued character (via a fixed embedding $\iota_\infty:\overline{\Q}\hookrightarrow\bC^\times$) defined on the group of fractional ideals of $K$ coprime to $\mathfrak{f}_\chi$, with a slight abuse of notation we also denote by $\chi$ the character of $Z(\mathfrak{C})$ defined by the rule
\[
\chi(\sigma_\mathfrak{l}^{-1})=\iota_p(\chi(\mathfrak{l}))
\]
for all primes $\mathfrak{l}\nmid\mathfrak{C}$, where $\sigma_\mathfrak{l}\in Z(\mathfrak{C})$ is the Artin symbol of $\mathfrak{l}$.

\begin{thm}\label{thm:katz}
There exists an element $\cL_{\pp,\mathfrak{C}}\in\cW\dBr{Z(\mathfrak{C})}$ such that for all Hecke characters $\chi$ 
of conductor $\mathfrak{f}_\chi\mid \mathfrak{C}p^\infty$ and infinity type $(k+j,-j)$ with $k\geq 1$, $j\geq 0$, we have
\begin{equation}\label{eq:Katz}
\cL_{\pp,\mathfrak{C}}(\chi)=\biggl(\frac{\Omega_p}{\Omega_\infty}\biggr)^{k+2j}\cdot\Gamma(k+j)\cdot\biggl(\frac{2\pi}{\sqrt{D_K}}\biggr)^j\cdot\mathcal{E}_\pp(\chi)
\cdot L^{(p\mathfrak{C})}(\chi,0),\nonumber
\end{equation}
where $\mathcal{E}_\pp(\chi)$ is the modified Euler factor
\[
\mathcal{E}_\pp(\chi)=\frac{L(0,\chi_\pp)}{\varepsilon(0,\chi_\pp)\cdot L(1,\chi_\pp^{-1})},
\]
and $L^{(p\mathfrak{C})}(\chi,s)$ is the $L$-function of $\chi$ deprived from the Euler factors at the primes dividing $p\mathfrak{C}$. Moreover, we have the functional equation 
\[
\cL_{\pp,\mathfrak{C}}(\chi)=\cL_{\pp,\overline{\mathfrak{C}}}((\chi^{\cc})^{-1}\mathbf{N}^{-1}),
\]
where $\chi^{\cc}$ is the composition of $\chi$ with the action of the non-trivial element $\cc\in{\rm Aut}(K/\Q)$, and the equality is up to a $p$-adic unit.

\end{thm}

\begin{proof}
See \cite{Katz49}, \cite{HT-ENS}; our formulation of the interpolation property follows \cite[Prop.~4.9]{hsieh-crelle-mu} most closely. The functional equation is shown in \cite[Thm.~II.6.4]{deshalit}. 
\end{proof}

\subsection{Triple product $p$-adic $L$-function}\label{subsec:triple}

Let $\cR$ be a normal domain finite flat over 
\[
\Lambda:=\mathscr{O}\dBr{1+p\Z_p},
\] 
where $\mathscr{O}$ is the ring of integers of a finite extension of $\Q_p$. For a positive integer $N$ with $p\nmid N$ and a Dirichlet character $\chi:(\Z/Np\Z)^\times\rightarrow\mathscr{O}^\times$, we denote by $S^o(N,\chi,\cR)\subset\cR\dBr{q}$ the space of ordinary $\cR$-adic cusp forms of tame level $N$ and branch character $\chi$ as defined in \cite[\S{3.1}]{hsieh-triple}. 

Denote by $\mathfrak{X}_\cR^+\subset{\rm Spec}\,\cR(\overline{\Q}_p)$ the set of \emph{arithmetic points} of $\cR$, consisting of the ring homomorphisms $Q:\cR\rightarrow\overline{\Q}_p$ such that $Q\vert_{1+p\Z_p}$ is given by $z\mapsto z^{k_Q}\epsilon_Q(z)$ for some $k_Q\in\Z_{\geq 2}$ called the \emph{weight of $Q$} and $\epsilon_Q(z)\in\mu_{p^\infty}$. As in  \cite[\S{3.1}]{hsieh-triple}, we say that $\boldsymbol{f}=\sum_{n=1}^\infty a_n(\bfff)q^n\in S^o(N,\chi,\cR)$ is a \emph{primitive Hida family} if for every $Q\in\mathfrak{X}_\cR^+$ the specialisation $\boldsymbol{f}_Q$ gives the $q$-expansion of an ordinary $p$-stabilised newform of weight $k_Q$ and tame conductor $N$. Attached to such $\bfff$ we let $\mathfrak{X}_{\cR}^{\rm cls}$ be the set of ring homomorphisms $Q$ as above with $k_Q\in\Z$ such that $\bfff_Q$ is the $q$-expansion of a classical modular form. Note that $\mathfrak{X}_{\cR}^{\rm cls}$ contains $\mathfrak{X}_{\cR}^{+}$ by Hida's results, but it can also contain points $Q$ --- of special relevance for this paper --- for which $\bfff_Q$ is a classical \emph{weight one} form.

For $\bfff$ a primitive Hida family, we let 
\[
\rho_{\bfff}:G_\Q\rightarrow{\rm Aut}_{\cR}(V_\bfff)\simeq{\rm GL}_2(\cR)
\]
denote the associated Galois representation, which  here we take to be the \emph{dual} of that in \cite[\S{3.2}]{hsieh-triple}; in particular, the determinant of $\rho_\bfff$ is $\chi_\cR\cdot\varepsilon_{\rm cyc}$ in the notations of \emph{loc.\,cit.}, where $\varepsilon_{\rm cyc}$ is the $p$-adic cyclotomic character. Note also that the $\rho_\bfff$ in \emph{loc.\,cit.} is valued in 
${\rm Frac}(\cR)$, but it is well-known that when the residual representation $\bar\rho_{\bfff}$ is absolutely irreducible, one may find a free $\cR$-module $V_\bfff$ of rank $2$ realising the same Galois representation after extension of scalars. 
There may be different $V_\bfff$ giving rise to the same rational $\rho_\bfff$ after extension of scalars; for now we take any such $V_\bfff$ (assuming $\bar{\rho}_\bfff$ to be  absolutely irreducible), and later in $\S\ref{subsec:dec-S}$ we specify a particular choice well-suited to our applications.

By \cite[Thm.~2.2.2]{wiles88}, restricted to $G_{\Q_p}$ the Galois representation $V_\bfff$ fits into a short exact sequence
\[
0\rightarrow V_\bfff^+\rightarrow V_\bfff\rightarrow V_\bfff^-\rightarrow 0,
\]
where the quotient $V_\bfff^-$ is free of rank one over $\cR$, with the $G_{\Q_p}$-action given by the unramified character sending an arithmetic Frobenius $\sigma_p$ to $a_p(\bfff)$. 

Let
\[
(\bff,\bfg,\bfh)\in S^o(N_\varphi,\chi_\varphi,\cR_\varphi)\times S^o(N_g,\chi_g,\cR_g)\times S^o(N_h,\chi_h,\cR_h)
\]
be a triple of primitive Hida families with 
\begin{equation}\label{eq:a}
\textrm{$\chi_\varphi\chi_g\chi_h=\omega^{2a}$ for some $a\in\Z$,}
\end{equation} 
where $\omega$ is the Teichm\"uller character. Put 
\[
\mathcal{R}=\cR_\varphi\hat\otimes_{\mathscr{O}}\cR_g\hat\otimes_{\mathscr{O}}\cR_h,
\] 
which is a finite extension of the three-variable Iwasawa algebra $\Lambda\hat\otimes_{\mathscr{O}}\Lambda\hat\otimes_{\mathscr{O}}\Lambda$, 
and let
\begin{equation}
\begin{aligned}
\mathfrak{X}_{\mathcal{R}}^\bff:=\{(Q_0,Q_1,Q_2)&\in\mathfrak{X}_{\cR_\varphi}^+\times\mathfrak{X}_{\cR_g}^{\rm cls}\times\mathfrak{X}_{\cR_h}^{\rm cls}\\
&\quad\colon\textrm{$k_{Q_0}\geq k_{Q_1}+k_{Q_2}$ and $k_{Q_0}\equiv k_{Q_1}+k_{Q_2}\;({\rm mod}\;2)$}\}
\end{aligned}\nonumber
\end{equation}
be the weight space for $\mathcal{R}$ in the so-called \emph{$\bff$-unbalanced range}. 

Let $\mathbf{V}=V_\bff\hat\otimes_{\mathscr{O}}V_{\bfg}\hat\otimes_{\mathscr{O}}V_{\bfh}$ be the triple tensor product Galois representation attached to $(\bff,\bfg,\bfh)$, and writing $\det\mathbf{V}=\mathcal{X}^2\varepsilon_{\rm cyc}$ (as is possible by \eqref{eq:a}), set  
\begin{equation}\label{eq:crit-twist}
\Vdag:=\mathbf{V}\otimes\mathcal{X}^{-1},
\end{equation}
which is a self-dual twist of $\mathbf{V}$. Define the rank four $G_{\Q_p}$-invariant subspace $\mathscr{F}_p^\bff(\Vdag)\subset\Vdag$ by
\begin{equation}\label{eq:unb-intro}
\mathscr{F}_p^\bff(\Vdag):=V_{\bff}^+\hat\otimes_{\mathscr{O}}V_\bfg\hat\otimes_{\mathscr{O}}V_{\bfh}\otimes\mathcal{X}^{-1}.
\end{equation}
For every $\underline{Q}=(Q_0,Q_1,Q_2)\in\mathfrak{X}_{\mathcal{R}}^\bff$ we denote by $\mathscr{F}_p^\bff(\VQdag)\subset\VQdag$  the corresponding specialisation. Finally, for every rational prime $\ell$ denote by $\varepsilon_\ell(\VQdag)$ the epsilon factor attached to the restriction of $\VQdag$ to $G_{\Q_\ell}$ as in \cite[p.\,21]{tate-background}, and assume that 
\begin{equation}\label{eq:+1}
\textrm{for some $\underline{Q}\in\mathfrak{X}_{\mathcal{R}}^\bff$, we have $\varepsilon_\ell(\VQdag)=+1$ for all primes $\ell\mid N_\bff N_\bfg N_\bfh$.}
\end{equation}
As explained in \cite[\S{1.2}]{hsieh-triple}, it is known that condition (\ref{eq:+1}) is independent of $\underline{Q}$, and it implies that the sign in the functional equation for the triple product $L$-function 
\[
L(\VQdag,s)
\] 
(relating its values at $s$ and $-s$) is $+1$ for all $\underline{Q}\in\mathfrak{X}_{\mathcal{R}}^\bff$. 

For the next statement, we refer the reader to $\S\ref{subsec:congr}$ below for a review of the congruence ideal associated with a primitive Hida family.


\begin{thm}\label{thm:hsieh-triple}
Let $(\bff,\bfg,\bfh)$ be a triple of primitive Hida families as above satisfying conditions (\ref{eq:a}) and (\ref{eq:+1}). Assume in addition that:
\begin{itemize}
\item[(i)] $\gcd(N_\varphi,N_g,N_h)$ is square-free,
\item[(ii)] the residual representation $\bar{\rho}_{\bff}$ is absolutely irreducible and $p$-distinguished,
\end{itemize}
and fix a generator $\eta_{\bff}$ of the congruence ideal of $\bff$. Then there exists a unique element 
\[
\mathscr{L}_p^{\bff}(\bff,\bfg,\bfh)\in\mathcal{R}
\]
such that for all $\underline{Q}=(Q_0,Q_1,Q_2)\in\mathfrak{X}_{\mathcal{R}}^{\bff}$ of weight $(k_0,k_1,k_2)$ with $\epsilon_{Q_0}=1$ we have
\[
(\mathscr{L}_p^\bff(\bff,\bfg,\bfh)(\underline{Q}))^2=\Gamma_{\VQdag}(0)\cdot\frac{L(\VQdag,0)}{(\sqrt{-1})^{2k_{0}}\cdot\Omega_{\bff_{Q_0}}^2}\cdot\mathcal{E}_p(\mathscr{F}_p^{\bff}(\VQdag))\cdot\prod_{\ell\in\Sigma_{\rm exc}}(1+\ell^{-1})^2,
\]
where:
\begin{itemize}
\item $\Gamma_{\VQdag}(0)=\Gamma_{\bC}(c_{\underline{Q}})\Gamma_\bC(c_{\underline{Q}}+2-k_1-k_2)\Gamma_\bC(c_{\underline{Q}}+1-k_1)\Gamma_\bC(c_{\underline{Q}}+1-k_2)$, with 
\[
c_{\underline{Q}}=(k_0+k_1+k_2-2)/2
\] 
and $\Gamma_\bC(s)=2(2\pi)^{-s}\Gamma(s)$;
\item $\Omega_{\bff_{Q_0}}$ is the canonical period
\[
\Omega_{\bff_{Q_0}}:=(-2\sqrt{-1})^{k_0+1}\cdot\frac{\Vert\bff_{Q_0}^\circ\Vert_{\Gamma_0(N_{\varphi})}^2}{\imath_p(\eta_{\bff_{Q_0}})}\cdot\Bigl(1-\frac{\chi_{\varphi}'(p)p^{k_0-1}}{\alpha_{Q_0}^2}\Bigr)\Bigl(1-\frac{\chi_{\varphi}'(p)p^{k_0-2}}{\alpha_{Q_0}^2}\Bigr),
\]
with $\bff_{Q_0}^\circ\in S_{k_0}(N_{\varphi})$ the newform of conductor $N_\varphi$ associated with $\bff_{Q_0}$, $\chi_\varphi'$ the prime-to-$p$ part of $\chi_\varphi$, and $\alpha_{Q_0}$ the specialisation of $a_p(\bff)\in\cR_\varphi^\times$ at $Q_0$;
\item $\mathcal{E}_p(\mathscr{F}_p^{\bff}(\VQdag))$ is the modified $p$-Euler factor
\[
\mathcal{E}_p(\mathscr{F}_p^{\bff}(\VQdag)):=\frac{L_p(\mathscr{F}_p^{\bff}(\VQdag),0)}{\varepsilon_p(\mathscr{F}_p^{\bff}(\VQdag))\cdot L_p(\VQdag/\mathscr{F}_p^{\bff}(\VQdag),0)}\cdot\frac{1}{L_p(\VQdag,0)},
\]
\end{itemize}
and $\Sigma_{\rm exc}$ is an explicitly defined subset of the prime factors of $N_\varphi N_g N_h$, \cite[p.~416]{hsieh-triple}.
\end{thm}

\begin{proof}
This is \cite[Thm.~A]{hsieh-triple}. 
\end{proof}

\begin{rem}
For simplicity,  
we have stated the interpolation property of $\mathscr{L}_p^\bff(\bff,\bfg,\bfh)$ only for $\underline{Q}$ with $\epsilon_{Q_0}=1$, as this will suffice for our purposes; see \cite[Thm.~A]{hsieh-triple} for the interpolation property for all $\underline{Q}\in\mathfrak{X}_{\mathcal{R}}^\bff$.
\end{rem}

\subsection{Congruence ideal}\label{subsec:congr}

Let $\bfff\in S^o(N_f,\chi_f,\cR)$ be a primitive Hida family defined over $\cR$. Associated with $\bfff$ there is a $\cR$-algebra homomorphism 
\[
\lambda_{\bfff}:\bT(N_f,\cR)\rightarrow\cR
\] 
where $\bT(N_f,\cR)$ is the Hecke algebra acting on $\oplus_\chi S^o(N_f,\chi,\cR)$, with $\chi$ running over the characters of $(\Z/pN_f\Z)^\times$. Let $\bT_{\mathfrak{m}}$ be the local component of $\bT(N_f,\cR)$ through which $\lambda_{\bfff}$ factors, and following \cite{hida-AJM-modules} define the \emph{congruence ideal} $C(\bfff)$ of $\bfff$ by
\[
C(\bfff):=\lambda_{\bfff}({\rm Ann}_{\bT_\mathfrak{m}}({\rm ker}\,\lambda_\bfff))\subset\cR.
\] 
When the residual representation $\bar{\rho}_{\bfff}$ is absolutely  irreducible and $p$-distinguished, it follows from the results of \cite{Fermat-Wiles} and \cite{hida-AJM-modules} that $C(\bfff)$ is generated by a nonzero element $\eta_{\bfff}\in\cR$.

\section{Factorisation of $p$-adic triple product $L$-function}

In this section we relate the triple product $p$-adic $L$-function attached to triples of Hida families with CM by $K$ to a product of anticyclotomic Katz $p$-adic $L$-functions. 

\subsection{Hida families with CM}\label{subsec:CM-hida}

We review the construction of CM Hida families following the exposition in \cite[\S{8.1}]{hsieh-triple}. 
Since it will suffice for our purposes, we assume that  the class number $h_K=\vert{\rm Pic}(\cO_K)\vert$ of $K$ is coprime to $p$. Let $K_\infty$ be the unique $\Z_p^2$-extension of $K$, and denote by $K_{\pp^\infty}$ the maximal subfield of $K_\infty$ unramified outside $\pp$. Put 
\[
\Gamma_\infty:={\rm Gal}(K_\infty/K)\simeq\Z_p^2,\quad\quad
\Gamma_{\pp^\infty}:={\rm Gal}(K_{\pp^\infty}/K)\simeq\Z_p.
\]

Recall that for every ideal $\mathfrak{C}\subset\cO_K$ we let $K(\mathfrak{C})$ be the ray class field of $K$ of conductor $\mathfrak{C}$. Thus $K_{\pp^\infty}$ is also the maximal $\Z_p$-extension inside $K(\pp^\infty)$. By our assumption on $h_K$, the restriction of the 
Artin map to $K_\pp^\times$ induces an isomorphism $1+p\Z_p\simeq\Gamma_{\pp^\infty}$, where we identified $\Z_p^\times$ and $\cO_{K_\pp}^\times$ by the embedding $\iota_p$. Denote by $\gamma_\pp$ the topological generator of $\Gamma_{\pp^\infty}$ corresponding to $1+p$ under this isomorphism, and for each variable $S$ let $\Psi_S:\Gamma_\infty\rightarrow\Z_p\dBr{S}^\times$ be the universal character 
given by
\begin{equation}\label{eq:univ-Psi}
\Psi_S(\sigma)=(1+S)^{l(\sigma)},
\end{equation}
where $l(\sigma)\in\Z_p$ is such that $\sigma\vert_{K_{\pp^\infty}}=\gamma_\pp^{l(\sigma)}$. Fix $\mathfrak{C}$ coprime to $p$, and for any finite order character $\xi:G_K\rightarrow\mathscr{O}^\times$ of conductor $\mathfrak{C}\pp^r$, $r\geq 0$, put
\begin{equation}\label{eq:def-CM}
\boldsymbol{\theta}_\xi(S)(q)=\sum_{(\fa,\pp\mathfrak{C})=1}\xi(\sigma_{\fa})\Psi^{-1}_{\mathbf{u}^{-1}(1+S)-1}(\sigma_\fa)q^{\mathbf{N}(\fa)}\in\mathscr{O}\dBr{S}\dBr{q},
\end{equation}
where $\mathbf{u}=1+p$ and $\sigma_\fa\in{\rm Gal}(K(\mathfrak{C}\pp^\infty)/K)$ is the Artin symbol of $\mathfrak{a}$. 
Then $\boldsymbol{\theta}_\xi(S)$ is a primitive Hida family %
(called the primitive \emph{CM Hida family} with branch character $\xi$) defined over $\mathscr{O}\dBr{S}$ of level $D_K\mathbf{N}(\mathfrak{C})$  and tame character $(\xi\circ\mathscr{V})\eta_{K/\Q}\omega^{-1}$, where 
\[
\mathscr{V}:G_\Q^{\rm ab}\rightarrow G_K^{\rm ab}
\] 
is the transfer map and $\eta_{K/\Q}$ is the quadratic character associated to $K/\Q$.

\subsection{Congruence ideal of CM Hida families}\label{subsec:congr-CM}

Let $\bff=\boldsymbol{\theta}_{\xi}(S)$ be a primitive CM Hida family defined over $\mathscr{O}\dBr{S}$. In this section we recall the characterisation of the congruence ideal of $\bff$ that follows from its relation with the anticyclotomic main conjecture for the underlying imaginary quadratic field $K$. 

Let $N_\varphi=D_K\mathbf{N}(\mathfrak{C})$ be the tame conductor of $\bff$. Assume that $\mathfrak{C}$ is coprime to $p$ and fixed under the action of complex conjugation, so the non-trivial automorphism $\cc\in{\rm Aut}(K/\Q)$ acts on $Z(\mathfrak{C})={\rm Gal}(K(\mathfrak{C}p^\infty)/K)$. Let $\Delta_{\mathfrak{C}}$ be the torsion subgroup of $Z(\mathfrak{C})$, and put $\Gamma_K:=Z(\mathfrak{C})/\Delta_{\mathfrak{C}}\simeq\Z_p^2$, which is identified with the Galois group of the unique $\Z_p^2$-extension $K_\infty/K$. Fix a decomposition
\begin{equation}\label{eq:dec-Gamma}
Z(\mathfrak{C})\simeq\Delta_{\mathfrak{C}}\times\Gamma_K,
\end{equation}
and note that $\Delta_\mathfrak{C}$ has order prime-to-$p$ by our assumption that $p\nmid h_K$.

Let $Z(\mathfrak{C})^+$ 
be the maximal subgroup of $Z(\mathfrak{C})$ 
fixed by the action of $\cc$. Put $Z(\mathfrak{C})^-=Z(\mathfrak{C})/Z(\mathfrak{C})^+$, and denote by $\pi:Z(\mathfrak{C})\rightarrow Z(\mathfrak{C})^-$ the natural projection. We also have a decomposition (which we fix to be compatible with \eqref{eq:dec-Gamma} under $\pi$) 
\[
Z(\mathfrak{C})^-\simeq\Delta_{\mathfrak{C}}^-\times\Gamma^-,
\]
where $\Delta_{\mathfrak{C}}^-$ is the torsion subgroup of $\Delta_{\mathfrak{C}}$ and $\Gamma^-$ is the eigenspace of $\Gamma_K$ where $\cc$ acts as inversion.  

Let $\mathbb{F}$ be the residue field of $\mathscr{O}$, and denote by $\bar{\xi}:G_K\rightarrow\mathbb{F}^\times$ the reduction of $\xi$. Then from \eqref{eq:def-CM} we see that the residual representation $\bar\rho_{\bff}$ satisfies 
\begin{equation}\label{eq:residual-CM}
\bar\rho_{\bff}\simeq{\rm Ind}_K^\bQ\bar{\xi}^{-1}
\end{equation} 
Put $\bar{\xi}^-=\bar{\xi}^{\cc-1}$, which we shall view as being $\mathscr{O}$-valued by the Teichm\"uller lift.  
Since $\bar{\xi}^-$ has order prime to $p$, its composition with $\pi$ factors through $\Delta_{\mathfrak{C}}$, but its prime-to-$p$ conductor may be a proper divisor of $\mathfrak{C}$. Let $\mathfrak{c}$ be the prime-to-$p$ conductor of $\bar{\xi}^-$ viewed as a character of $\Delta_{\mathfrak{C}}$ via $\pi$ (i.e., $\mathfrak{c}$ is the maximal divisor of $\mathfrak{C}$ such that $\bar{\xi}^-\pi$ factors through the quotient map $\Delta_{\mathfrak{C}}\twoheadrightarrow\Delta_{\mathfrak{c}}$), put $\Gamma_K^{\cc-1}=\{\gamma^{\cc-1}\colon\gamma\in\Gamma_K\}$, and (upon enlarging $\mathcal{W}$ if necessary) let $\mathcal{L}_{\pp,\bar{\xi}^-}^{(\cc-1)}\in\cW\dBr{\Gamma_K^{\cc-1}}$ denote the image of the Katz $p$-adc $L$-function $\mathcal{L}_{\pp,\mathfrak{c}}$ under the composition
\[
\cW\dBr{Z(\mathfrak{c})}\rightarrow\cW\dBr{\Gamma_K}\rightarrow\cW\dBr{\Gamma_K^{\cc-1}},
\]
where the first arrow is the projection defined by $\psi^-$ (viewed as a primitive character on $\Delta_{\mathfrak{c}}$) and the second arrow is given by $\gamma\mapsto\gamma^{\cc-1}$ for $\gamma\in\Gamma_K$. Since $p$ is odd, the map $\gamma^{\cc-1}\mapsto(\gamma\vert_{K_{\pp^\infty}})^{1/2}$ yields an isomorphism $\Gamma_K^{\cc-1}\simeq\Gamma_{\pp^\infty}$. Upon choosing a topological generator $\gamma_\pp\in\Gamma_{\pp^\infty}$ as in $\S\ref{subsec:CM-hida}$, we shall thus view $\mathcal{L}_{\pp,\bar{\xi}^-}^{(\cc-1)}$ as an element in $\cW\dBr{S}$.

\begin{lem}\label{lem:HT}
Let $\bff=\boldsymbol{\theta}_{\xi}(S)$ be a primitive CM Hida family as in \eqref{eq:def-CM}.  
Suppose that
$\bar{\xi}^-\vert_{G_{K_\pp}}\neq 1$,  
where $G_{K_\pp}\subset G_K$ is a decomposition group at $\pp$. Then the congruence ideal $C(\bff)$ is generated by $\mathcal{L}_{\pp,\bar{\xi}^-}^{(\cc-1)}$ in $\cW\dBr{S}\otimes_{\Z_p}\Q_p$.
\end{lem}

\begin{proof}
The assumption on $\bar\xi^-$ implies that $\bar{\rho}_{\bff}$ is absolutely irreducible and $p$-distinguished, and so the congruence ideal $C(\bff)\subset\mathscr{O}\dBr{S}$ is known to be principal. Letting $\eta_\bff\in\cO\dBr{S}$ be a generator, from \cite[Thm.\,I]{HT-ENS} and \cite[Thm.\,0.3]{HT-117}  we have the divisibilities 
\begin{equation}\label{eq:HT}
\mathcal{L}_{\pp,\bar{\xi}^-}^{(\cc-1)}\mid\eta_\bff\mid\mathcal{F}_{\pp,\bar{\xi}^-}^{(\cc-1)}
\end{equation}
in $\cW\dBr{S}\otimes_{\Z_p}\Q_p$ and $\mathscr{O}\dBr{S}\otimes_{\bZ_p}\bQ_p$, respectively, where $\mathcal{F}_{\pp,\bar{\xi}^-}^{(\cc-1)}\in\mathscr{O}\dBr{S}$ is a characteristic power series for the Pontryagin dual of the $\bar{\xi}^-$-isotypic component of ${\rm Gal}(M_\infty^-/K_{\mathfrak{c},\infty}^-)$, the Galois group of the maximal pro-$p$ abelian extension of $K_{\mathfrak{c},\infty}^-:=K(\mathfrak{c}p^\infty)^{Z(\mathfrak{c})^+}$ unramified outside $\pp$. (For the second divisibility in \eqref{eq:HT}, note that the local non-triviality of $\bar{\xi}^-$ implies that ${\rm ord}_P(\mathcal{L}_{\pp,\bar{\xi}^-}^{(\cc-1)})=0$ for all ``trivial zero primes'' $P\subset\cW\dBr{S}$ in the sense of \cite{HT-117}.) Since Rubin's proof \cite{rubin-IMC} of the main conjecture for $K$  yields  $(\mathcal{F}_{\pp,\bar{\xi}^-}^{(\cc-1)})=(\mathcal{L}_{\pp,\bar{\xi}^-}^{(\cc-1)})$, together with \eqref{eq:HT}  
the result follows.
%
\end{proof}

\begin{rem}
As explained in \cite[\S{4.5}]{ACR2}, under some further hypotheses on the branch character $\bar{\xi}^-$, it follows from Hida's proof \cite{hida-coates} on the anticyclotomic main conjecture for $K$, refining the results  of  \cite{HT-ENS,HT-117}, that $C(\bff)$ is \emph{integrally} generated by $h_K\cdot\mathcal{L}_{\pp,\bar{\xi}^-}^{(\cc-1)}$, where $h_K$ is the class number of $K$. 
\end{rem}

\subsection{Proof of the factorisation}\label{subsec:factor-L}

Now suppose $(\lambda_0,\lambda_1,\lambda_2)$ are Hecke characters of $K$ with:
\begin{itemize}
\item $\lambda_0$ of infinity type $(-1,0)$ and conductor $\mathfrak{C}_0$,
\item $\lambda_1,\lambda_2$ ray class characters of conductor $\mathfrak{C}_1,\mathfrak{C}_2$.
\end{itemize}
We assume that $\mathfrak{C}_i$ is coprime to $p$ for $i=0,1,2$, and that the central character $\chi_{\lambda_i}$ are such that
\begin{equation}\label{eq:sd-epsK}
\chi_{\lambda_0}\chi_{\lambda_1}\chi_{\lambda_2}=\eta_{K/\Q}.
\end{equation}
This condition implies that each of the products 
\begin{equation}\label{eq:sd-lambda}
\lambda_0^{-1}\lambda_1^{-1}\lambda_2^{-1},\quad \lambda_0^{-1}\lambda_1^{-\cc}\lambda_2^{-\cc},\quad 
\lambda_0^{-1}\lambda_1^{-\cc}\lambda_2^{-\cc},\quad 
\lambda_0^{-1}\lambda_1^{-\cc}\lambda_2^{-1}
\end{equation}
is a \emph{self-dual} character, where we say that $\phi$ is self-dual if $\phi^\cc=\phi^{-1}\mathbf{N}^{-1}$ and $\chi_\phi=\eta_{K/\bQ}$ (see also \cite[Rem.~3.7]{bdp3} for the above claim). For such $\phi$, the Hecke $L$-function $L(\phi,s)$ is self-dual, with a functional equation relating its values at $s$ and $-s$.

Note that the specialization of the character $\Psi_S$ of \eqref{eq:univ-Psi} to $S=\mathbf{u}-1$ descends to an isomorphism
\[
\Psi_{\mathbf{u}-1}:\Gamma_{\pp^\infty}\xrightarrow{\sim}1+p\Z_p,
\]
which we shall see as the $p$-adic avatar of a Hecke character $\psi_0$ of infinity type $(1,0)$ and conductor $\pp$. The character $\lambda_0$ can be written as 
\[
\lambda_0=\xi_0\psi_0^{-1}
\]
with $\xi_0$ a ray class character of $K$ of conductor dividing $\mathfrak{C}_0\pp$, and with this we put
\[
\boldsymbol{\theta}_{\lambda_0}(S_0):=\boldsymbol{\theta}_{\xi_0}(\mathbf{u}^{-1}(1+S_0)-1)\in\scrO\dBr{S_0}\dBr{q},
\]
and consider the triple of CM Hida families
\[
(\bff,\bfg,\bfh)=(\boldsymbol{\theta}_{\lambda_0}(S_0),\boldsymbol{\theta}_{\lambda_1}(S_1),\boldsymbol{\theta}_{\lambda_2}(S_2)),
\] 
which satisfies conditions \eqref{eq:a} and \eqref{eq:+1}. By construction, the specialisation of $\bff$ at the weight $2$ arithmetic point $Q_0\in\mathfrak{X}_{\mathscr{O}\dBr{S_0}}^+$ given by $S_0=\mathbf{u}^2-1$ gives the ordinary $p$-stabilisation $\varphi$ of $\theta(\psi_0)$ with $U_p$-eigenvalue $\lambda_0(\ppbar)$.
%
Let
\begin{equation}\label{eq:sp}
\mathcal{R}\simeq\mathscr{O}\dBr{S_0,S_1,S_2}\rightarrow\mathscr{O}\dBr{S_1,S_2}
\end{equation}
be the specialisation map at $Q_0$, and denote by
\[
\mathscr{L}_{p}^\varphi(\varphi,\bfg,\bfh)(S_1,S_2)\in\mathscr{O}\dBr{S_1,S_2}
\] 
the image of the triple product $p$-adic $L$-function $\mathscr{L}_{p}^{\bff}(\bff,\bfg,\bfh)\in\mathcal{R}$ of Theorem~\ref{thm:hsieh-triple} under this map. On the other hand, for $\chi$ a Hecke character of $K$ of prime-to-$p$ conductor $\mathfrak{C}$, we denote by $\mathcal{L}_{\pp,\chi}^-$ the image of ${\rm Tw}_{\chi^{-1}}(\mathcal{L}_{\pp,\mathfrak{C}})$ under the natural projection $\cW\dBr{Z(\mathfrak{C})}\rightarrow\cW\dBr{\Gamma^-}$, where 
\[
{\rm Tw}_{\chi^{-1}}:\cW\dBr{Z(\mathfrak{C})}\rightarrow\cW\dBr{Z(\mathfrak{C})}
\]
is the $\cW$-linear isomorphism given by $\gamma\mapsto\chi^{-1}(\gamma)\gamma$ for $\gamma\in Z(\mathfrak{C})$. As usual, upon choosing a topological generator $\gamma^-\in\Gamma^-$, we shall identify $\cW\dBr{\Gamma^-}$ with the power series ring $\cW\dBr{W}$ via $\gamma^-\mapsto 1+W$. Finally, denote by $\lambda\mapsto\lambda^\iota$ the involution of $\cW\dBr{\Gamma^-}$ given by $\gamma\mapsto\gamma^{-1}$ for $\gamma\in\Gamma^-$.

\begin{prop}\label{prop:factor-L}
Let $(\varphi,\bfg,\bfh)=(\theta(\lambda_0),\boldsymbol{\theta}_{\lambda_1}(S_1),\boldsymbol{\theta}_{\lambda_2}(S_2))$ be a CM triple as above, and suppose in addition that $\lambda_0(\pp)\not\equiv\lambda_0(\ppbar)\pmod{p}$, and put
\[
W_1=\mathbf{u}^{-1}(1+S_1)^{1/2}(1+S_2)^{1/2}-1,\quad
W_2=(1+S_1)^{1/2}(1+S_2)^{-1/2}-1.
\]
Then we have the factorisation
\begin{align*}
\mathscr{L}_{p}^{\varphi}(\varphi,\bfg,\bfh)^2(S_1,S_2)&=\mathbf{w}\cdot\cL^-_{\pp,\lambda_0\lambda_1\lambda_2}(W_1)^\iota\cdot\cL_{\pp,\lambda_0\lambda_1^{\cc}\lambda_2^{\cc}}^-(W_1)\\
&\quad\times\cL_{\pp,\lambda_0\lambda_1\lambda_2^{\cc}}^-(W_2)^\iota\cdot\cL_{\pp,\lambda_0\lambda_1^{\cc}\lambda_2}^-(W_2),
\end{align*}
where $\mathbf{w}\in\cW[1/p]^\times$.
\end{prop}

\begin{proof}
%
We begin by noting that the assumption  $\lambda_0(\pp)\not\equiv\lambda_0(\ppbar)\pmod{p}$ implies that the residual representation $\bar{\rho}_{\bff}$ is absolutely irreducible and $p$-distinguished; so condition (ii) in Theorem~\ref{thm:hsieh-triple} holds. Moreover, since the local representations associated to $\bfg,\bfh$ at the ramified primes are principal series, the calculation of local zeta integrals in \cite[\S{6}]{hsieh-triple} apply to our case, without the need for hypothesis (i) in  Theorem~\ref{thm:hsieh-triple} (we thank the referee for pointing out to us the possibility of dispensing with hypothesis (i) in our setting; see \cite[\S{6.1}]{hsieh-triple} for its use in the general case). Thus we know that $\mathscr{L}_p^\varphi(\varphi,\bfg,\bfh)$ satisfies the interpolation formula of Theorem~\ref{thm:hsieh-triple}, and by Lemma~\ref{lem:HT} (noting that $\lambda_0\equiv\xi_0\pmod{p}$), we can take $\mathcal{L}_{\pp,\bar{\xi}^-}^{(\cc-1)}$ as a generator of $C(\bff)$ over $\cW\dBr{S_0}\otimes_{\Z_p}\Q_p$ in our normalisation $\mathscr{L}_p^{\bff}(\bff,\bfg,\bfh)$.

Now, for $i=1,2$, let $\zeta_i$ be a primitive $p^{n_i}$-th root of unity with $n_i>0$, and put $Q_1=\zeta_1\zeta_2\mathbf{u}-1$, $Q_2=\zeta_1\zeta_2^{-1}\mathbf{u}-1$, so the specialisations $\bfg_{Q_1},  \bfh_{Q_2}$ are both CM forms of weight $1$. 
Let $\epsilon_i:\Gamma^-\rightarrow\mu_{p^\infty}$ be the finite order character given by $\epsilon_i(\gamma^-)=\zeta_i$ and put $\underline{Q}=(Q_0,Q_1,Q_2)$. Noting that for the above triple $(\varphi,\bfg,\bfh)$ the character $\mathcal{X}$ in \eqref{eq:crit-twist} is given by 
\[
\mathcal{X}=(\Psi_{T_1}^{1/2}\Psi_{T_2}^{1/2}\circ\mathscr{V}):G_\Q\rightarrow\Z_p\dBr{S_1,S_2}^\times,
\]
where $T_i=\mathbf{u}^{-1}(1+S_i)-1$, we find the decomposition of $p$-adic representations
\begin{align*}
\VQdag\otimes\Q_p&=
\Ind_K^\Q(\lambda_0^{-1}\lambda_1^{-1}\lambda_2^{-1}\epsilon_1^{-1})\oplus\Ind_K^\Q(\lambda_0^{-1}\lambda_1^{-\cc}\lambda_2^{-\cc}\epsilon_1)\\
&\quad\oplus\Ind_K^\Q(\lambda_0^{-1}\lambda_1^{-1}\lambda_2^{-\cc}\epsilon_2^{-1})
\oplus\Ind_K^{\Q}(\lambda_0^{-1}\lambda_1^{-\cc}\lambda_2^{-1}\epsilon_2);
\end{align*}
\begin{align*}
\mathscr{F}_p^{\varphi}(\VQdag)\otimes\Q_p
&=\lambda_{0,\pp}^{-1}\lambda_{1,\pp}^{-1}\lambda_{2,\pp}^{-1}\epsilon_{1,\pp}^{-1}\oplus\lambda_{0,\pp}^{-1}\lambda_{1,\pp}^{-\cc}\lambda_{2,\pp}^{-\cc}\epsilon_{1,\pp}\\
&\quad\oplus\lambda_{0,\pp}^{-1}\lambda_{1,\pp}^{-1}\lambda_{2,\pp}^{-\cc}\epsilon_{2,\pp}^{-1}\oplus\lambda_{0,\pp}^{-1}\lambda_{1,\pp}^{-\cc}\lambda_{2,\pp}^{-1}\epsilon_{2,\pp}.
\end{align*}
Thus the terms appearing in the interpolation formula of Theorem~\ref{thm:hsieh-triple} become:
\begin{equation}\label{eq:explicit}
\begin{aligned}
\Gamma_{\VQdag}(0)\cdot L(\VQdag,0)&=\pi^{-4}\cdot L(\lambda_0^{-1}\lambda_1^{-1}\lambda_2^{-1}\epsilon_1^{-1},0)\cdot L(\lambda_0^{-1}\lambda_1^{-\cc}\lambda_2^{-\cc}\epsilon_1,0)\\
&\quad\times L(\lambda_0^{-1}\lambda_1^{-1}\lambda_2^{-\cc}\epsilon_2^{-1},0)\cdot L(\lambda_0^{-1}\lambda_1^{-\cc}\lambda_2^{-1}\epsilon_2,0);
\end{aligned}
\end{equation}
\begin{equation}
\begin{aligned}\label{eq:Ep}
\mathcal{E}_p(\mathscr{F}_p^{\varphi}(\VQdag))&=
\mathcal{E}_\pp(\lambda_0^{-1}\lambda_1^{-1}\lambda_2^{-1}\epsilon_1^{-1})\cdot\mathcal{E}_\pp(\lambda_0^{-1}\lambda_1^{-\cc}\lambda_2^{-\cc}\epsilon_1)\\
&\quad\times\mathcal{E}_\pp(\lambda_0^{-1}\lambda_1^{-1}\lambda_2^{-\cc}\epsilon_2^{-1})\cdot\mathcal{E}_\pp(\lambda_0^{-1}\lambda_1^{-\cc}\lambda_2^{-1}\epsilon_2);
\end{aligned}
\end{equation}
\begin{equation}\label{eq:period}
\Omega_{\bff_{Q_0}}=(-2\sqrt{-1})^3\cdot\frac{\Vert\varphi\Vert_{\Gamma_0(N_\varphi)}^2}{\imath_p(\eta_{\bff_{Q_0}})}\cdot\biggl(1-\frac{\lambda_0(\pp)}{\lambda_0(\bar\pp)}\biggr)\biggl(1-\frac{\lambda_0(\pp)}{p\lambda_0(\bar\pp)}\biggr),
\end{equation}
and we note that we may ignore the term $\prod_{q\in\Sigma_{\rm exc}}(1+q^{-1})^2$ from Theorem~\ref{thm:hsieh-triple} since it only contributes a fixed power of $p$ independent of $Q_1,Q_2$.

On the other hand, from Hida's formula for the adjoint $L$-value \cite[Thm.~7.1]{HT-ENS} and Dirichlet's class number formula we obtain 
\[
\Vert\varphi\Vert^2_{\Gamma_0(N_\varphi)}=\frac{D_K^2}{2^4\pi^3}\cdot\frac{2\pi h_K}{w_K\sqrt{D_K}}\cdot L(\lambda_0\lambda_0^{-\cc},1).
\]
Noting that $L(\lambda_0\lambda_0^{-\cc},1)=L(\lambda_0^\cc\lambda_0^{-1},1)=L(\lambda_0^\cc\lambda_0^{-1}\mathbf{N}^{-1},0)$ and that $\lambda_0^\cc\lambda_0^{-1}\mathbf{N}^{-1}$ has infinity type $(2,0)$, and letting $\mathfrak{c}$ denote the prime-to-$\pp$ conductor of $\lambda_0^{\cc-1}$, the interpolation property in Theorem~\ref{thm:katz} can thus be rewritten as
\begin{align*}
\cL_{\pp,\mathfrak{c}}(\lambda_0^\cc\lambda_0^{-1}\mathbf{N}^{-1})&
=\biggl(\frac{\Omega_p}{\Omega_\infty}\biggr)^2\cdot\frac{2^3\pi^2}{\sqrt{D_K}^{3}}\\
&\quad\times\biggl(1-\frac{\lambda_0(\pp)}{\lambda_0(\bar\pp)}\biggr)\biggl(1-\frac{\lambda_0(\pp)}{p\lambda_0(\bar\pp)}\biggr)\cdot\frac{w_K}{h_K}\cdot\Vert\varphi\Vert_{\Gamma_0(N_\varphi)}^2\\
&=-\biggl(\frac{\Omega_p}{\Omega_\infty}\biggr)^2\cdot\frac{\pi^2\sqrt{-1}}{\sqrt{D_K}^3}\cdot\Omega_{\bff_{Q_0}}\cdot\eta_{\bff_{Q_0}}\cdot\frac{w_K}{h_K},
\end{align*}
using (\ref{eq:period}) for the second equality. By the functional equation for Katz's $p$-adic $L$-function (see Theorem~\ref{thm:katz})  
together with Lemma~\ref{lem:HT}, this shows that
\begin{equation}\label{eq:Omega-bis}
\frac{1}{\Omega_{\bff_{Q_0}}^2}=\mathbf{w}'\biggl(\frac{\Omega_p}{\Omega_\infty}\biggr)^4\cdot\frac{\pi^4}{D_K^{3}},
\end{equation}
where  
$\mathbf{w}'\in\cW[1/p]^\times$ is independent of $Q_1, Q_2$. Thus substituting $(\ref{eq:explicit}), (\ref{eq:Ep}), (\ref{eq:period})$ and $(\ref{eq:Omega-bis})$ into the interpolation formula for $\mathscr{L}_p^\varphi(\varphi,\bfg,\bfh)$ we thus arrive at
\begin{align*}
\mathscr{L}_p^\varphi(\varphi,\bfg,\bfh)^2(\zeta_1\zeta_2\mathbf{u}-1,\zeta_1\zeta_2^{-1}\mathbf{u}-1)&=\mathbf{w}'D_K^{-3}\cdot\cL_{\pp,\lambda_0\lambda_1\lambda_2}^-(\zeta_1^{-1}-1)\cdot\cL_{\pp,\lambda_0\lambda_1^{\cc}\lambda_2^{\cc}}^-(\zeta_1-1)\\
&\quad\times\cL_{\pp,\lambda_0\lambda_1\lambda_2^{\cc}}^-(\zeta_2^{-1}-1)\cdot\cL_{\pp,\lambda_0\lambda_1^{\cc}\lambda_2}^-(\zeta_2-1),
\end{align*}
for all non-trivial $p$-power roots of unity $\zeta_1, \zeta_2$,  and the result follows.
\end{proof}

\section{Selmer group decompositions}

In this section we define different Selmer groups attached to Hecke characters and triple products of modular forms, and 
 prove a decomposition mirroring the factorisation in Proposition~\ref{prop:factor-L}.

\subsection{Selmer groups for Hecke characters}\label{subsec:Sel-char}

Let $\nu$ be a Hecke character of $K$ with values in the ring of integers $\mathscr{O}$ of a finite extension $\Phi$ of $\Q_p$. Denote by $\mathscr{O}_{\nu}$ the free $\mathscr{O}$-module of rank $1$ on which $G_K$ acts via $\nu^{-1}$ 
and put 
\[
T_\nu=\mathscr{O}_\nu,\quad V_\nu=T_\nu\otimes_{\mathscr{O}}\Phi,
\quad A_\nu=V_\nu/T_\nu=T_\nu\otimes_\mathscr{O}(\Phi/\mathscr{O}).
\]
Let $\Sigma$ be any finite set of places of $K$ containing $\infty$ and the primes dividing $p$ or the conductor of $\nu$, and for any finite extension $F/K$ denote by $G_{F,\Sigma}$ the Galois group of the maximal extension of $F$ unramified outside the places above $\Sigma$.

\begin{defn}
Let $F/K$ be a finite extension, and for $v$ a prime of $F$ above $p$ put
\[
\rH^1_{\emptyset}(F_v,V_\nu)=\rH^1(F_v,V_\nu),\quad\rH^1_{0}(F_v,V_\nu)=\{0\}.
\]
For $(\cL_{\pp},\cL_{\ppbar})\in\{\emptyset,0\}^{\oplus 2}$ define the Selmer group ${\rm Sel}_{\cL_{\pp},\cL_{\ppbar}}(F,V_\nu)$ by
\begin{align*}
{\rm Sel}_{\cL_{\pp},\cL_{\ppbar}}(F,V_\nu)=
\ker\biggl(\rH^1(G_{F,\Sigma},V_\nu)&\rightarrow \prod_{v\vert\pp}\frac{\rH^1(F_v,V_\nu)}{\rH^1_{\cL_\pp}(F_v,V_\nu)}\times
\prod_{v\vert\ppbar}\frac{\rH^1(F_v,V_\nu)}{\rH^1_{\cL_{\ppbar}}(F_v,V_\nu)}\\
&\quad\times
\prod_{v\in\Sigma, v\nmid p\infty}\rH^1(F_v,V_\nu)\biggr).
\end{align*}
\end{defn}

\begin{rem}
In particular, if $\nu$ has infinity type $(-1,0)$, the Selmer group
\[
{\rm Sel}_{\emptyset,0}(F,V_\nu)=\ker\biggl(\rH^1(G_{F,\Sigma},V_\nu)\rightarrow 
\prod_{v\vert\ppbar}\rH^1(F_v,V_\nu)\times
\prod_{v\in\Sigma, v\nmid p\infty}\rH^1(F_v,V_\nu)\biggr)
\]
agrees with the Bloch--Kato Selmer group of $V_\nu$ (see e.g.  \cite[\S{1.1}]{AHsplit} or \cite[\S{1.2}]{Arnold}). 
(Therefore in that case ${\rm Sel}_{0,\emptyset}(F,V_{\nu^\cc})$ agrees with the Bloch--Kato Selmer group of $V_{\nu^\cc}$.) 
\end{rem}

For $\any\in\{\emptyset,0\}$, the local condition $\rH^1_\any(F_v,T_\nu)$ and $\rH^1_\any(F_v,A_\nu)$ are defined from the above by propagation, and using these we define ${\rm Sel}_{\cL_\pp,\cL_{\ppbar}}(F,T_\nu)$ and ${\rm Sel}_{\cL_\pp,\cL_{\ppbar}}(F,A_\nu)$ by the same recipe as before. Then, letting $K_\infty/K$ denote the anticyclotomic $\Z_p$-extension of $K$, we set
\begin{equation}\label{eq:Iw-Sel}
\begin{aligned}
{\rm Sel}_{{\cL_\pp,\cL_{\ppbar}}}(K_\infty,T_\nu)&=\varprojlim_n{\rm Sel}_{{\cL_\pp,\cL_{\ppbar}}}(K_n,T_\nu),\\
{\rm Sel}_{{\cL_\pp,\cL_{\ppbar}}}(K_\infty,A_\nu)&=\varinjlim_n{\rm Sel}_{{\cL_\pp,\cL_{\ppbar}}}(K_n,A_\nu).
\end{aligned}
\end{equation}

Let $\Lambda=\mathscr{O}\dBr{\Gamma^-}\simeq\mathscr{O}\dBr{W}$ be the anticyclotomic Iwasawa algebra, and denote by $\Psi_W:G_K\rightarrow\Lambda^\times$ the universal character given $g\mapsto[g\vert_{K_\infty}]$, where $\gamma\mapsto[\gamma]$ is the inclusion of $\Gamma^-$ intro $\Lambda^-$ as a group-like element. Then by Shapiro's lemma we have isomorphisms 
\[
\rH^1(K,T_\nu\otimes\Psi_W^{-1})\simeq\varprojlim_n\rH^1(K_n,T_\nu),\quad
\rH^1(K,A_\nu\otimes\Psi_W)\simeq\varinjlim_n\rH^1(K_n,A_\nu).
\]
In the following we let ${\rm Sel}_{\cL_\pp,\cL_{\ppbar}}(K,T_\nu\otimes\Psi_W^{-1})$ and ${\rm Sel}_{\cL_\pp,\cL_{\ppbar}}(K,A_\nu\otimes\Psi_W)$ be the Selmer groups corresponding to \eqref{eq:Iw-Sel} under these isomorphism, and let
\[
X_{\cL_\pp,\cL_{\ppbar}}(K,A_\nu\otimes\Psi_W)={\rm Hom}_{\Z_p}({\rm Sel}_{\cL_\pp,\cL_{\ppbar}}(K,A_\nu\otimes\Psi_W),\Q_p/\Z_p)
\] 
denote the Pontryagin dual of ${\rm Sel}_{\cL_\pp,\cL_{\ppbar}}(K,A_\nu\otimes\Psi_W)$.



We conclude this section by recalling the following result due to Agboola--Howard \cite{AHsplit} and Arnold \cite{Arnold} 
on the anticyclotomic Iwasawa main conjecture for $K$.

\begin{thm}\label{thm:AH}
Let $\nu$ be a Hecke character of $K$ of infinity type $(-1,0)$ satisfying that $\nu^\cc=\overline{\nu}$, where $\overline{\nu}$ is the complex conjugate of $\nu$, and suppose that $\nu$ has sign $+1$ in its functional equation. Then ${\rm Sel}_{\emptyset,0}(K,T_{\nu}\otimes\Psi_{W}^{-1})=0$, $X_{0,\emptyset}(K,A_{\nu^{\cc}}\otimes\Psi_{W}^{})$ is $\Lambda$-torsion, and
\[
{\rm char}_\Lambda\bigl(X_{0,\emptyset}(K,A_{\nu^{\cc}}\otimes\Psi_W^{})\bigr)=\bigl(\cL^-_{\pp,\nu}(W)\bigr)
\]
as ideals in $\Lambda_\cW\otimes_{\Z_p}\bQ_p$. 
\end{thm}

\begin{proof}
%
In the case where $\nu$ is the Hecke character associated with an elliptic curve $E/\bQ$ with CM by $K$, this is \cite[Thm.\,2.4.17]{AHsplit}. The general case is given in \cite[Thm.\,2.1]{Arnold}.
\end{proof}

\subsection{Selmer groups for triple products}\label{subsec:Sel-triple}

Let $(\bff,\bfg,\bfh)$ be a triple of primitive Hida families as in $\S\ref{subsec:triple}$ satisfying $\eqref{eq:a}$, and recall the self-dual twist $\Vdag=V_\bff\otimes V_\bfg\otimes V_\bfh\otimes\mathcal{X}^{-1}$ 
of the tensor product of their associated big Galois representations. 


\begin{defn}\label{def:local-p}
Put
\[
\mathscr{F}^{\rm bal}_p(\Vdag)=\mathscr{F}_p^2(\Vdag)
:=\bigl(V_{\bff}^+\otimes V_{\bfg}^+\otimes V_{\bfh}+V_{\bff}^+\otimes V_{\bfg}\otimes V_{\bfh}^++V_{\bff}\otimes V_{\bfg}^+\otimes V_{\bfh}^+\bigr)\otimes\mathcal{X}^{-1},
\]
and define the \emph{balanced local condition} $\rH^1_{\rm bal}(\Q_p,\Vdag)$ by 
\[
\rH^1_{\rm bal}(\Q_p,\Vdag):={\rm im}\bigl(\rH^1(\Q_p,\mathscr{F}_p^{\rm bal}(\Vdag))\rightarrow\rH^1(\Q_p,\Vdag)\bigr).
\]
Similarly, put $\mathscr{F}_p^{\bff}(\Vdag):=\bigl(V_{\bff}^+\otimes V_{\bfg}\otimes V_{\bfh}\bigr)\otimes\mathcal{X}^{-1}$ and define the \emph{$\bff$-unbalanced local condition} $\rH^1_{\unb}(\Q_p,\Vdag)$ by
\[
\rH^1_{\bff}(\Q_p,\Vdag):={\rm im}\bigl(\rH^1(\Q_p,\mathscr{F}_p^{\bff}(\Vdag))\rightarrow\rH^1(\Q_p,\Vdag)\bigr).
\]
\end{defn}

It is easy to see that the maps appearing in these definitions are injective, and in the following we shall use this to identify $\rH^1_{\any}(\Q_p,\Vdag)$ with $\rH^1(\Q_p,\mathscr{F}_p^\any(\Vdag))$ for $\any\in\{{\rm bal},\bff\}$.

\begin{defn}\label{def:Sel-bu}
Let $S$ be a finite set of primes containing $\infty$ and the primes dividing $pN_\varphi N_g N_h$, and let $G_{\bQ,S}$ be the Galois group of the maximal extension of $\Q$ unramified outside $S$. For every $\any\in\{{\rm bal},\bff\}$, define the Selmer group ${\rm Sel}^\any(\Q,\Vdag)$ by
\[
{\rm Sel}^\any(\Q,\Vdag):=\ker\biggl\{\rH^1(G_{\Q,S},\Vdag)\rightarrow\frac{\rH^1(\Q_p,\Vdag)}{\rH^1_\any(\Q_p,\Vdag)}\times\prod_{v\neq p}\rH^1(\Q_v^{\rm nr},\Vdag)\biggr\}.
\]
We call ${\rm Sel}^{\rm bal}(\Q,\Vdag)$ (resp. ${\rm Sel}^{\bff}(\Q,\Vdag)$) the \emph{balanced} (resp. \emph{$\bff$-unbalanced}) Selmer group.
\end{defn}

Let $\Adag={\rm Hom}_{\Z_p}(\Vdag,\mu_{p^\infty})$ and for $\any\in\{{\rm bal},\bff\}$ define $\rH^1_{\any}(\Q_p,\Adag)\subset\rH^1(\Q_p,\Adag)$ to be the orthogonal complement of $\rH^1_\any(\Q_p,\Vdag)$ under the local Tate duality
\[
\rH^1(\Q_p,\Vdag)\times\rH^1(\Q_p,\Adag)\rightarrow\Q_p/\Z_p.
\]
Similarly as above, we then define the balanced and $\bff$-unbalanced Selmer groups with coefficients in $\Adag$ by
\[
{\rm Sel}^\any(\Q,\Adag):=\ker\biggl\{\rH^1(G_{\Q,S},\Adag)\rightarrow\frac{\rH^1(\Q_p,\Adag)}{\rH_\any^1(\Q_p,\Adag)}\times\prod_{v\neq p}\rH^1(\Q_v^{\rm nr},\Adag)\biggr\},
\]
and let $X^\any(\Q,\Adag)={\rm Hom}_{\Z_p}({\rm Sel}^\any(\Q,\Adag),\Q_p/\Z_p)$ be the Pontryagin dual of ${\rm Sel}^\any(\Q,\Adag)$.

\subsection{Proof of the decompositions}\label{subsec:dec-S}

For a primitive Hida family $\bfff$, from now on we let $V_\bfff$ be the realisation of $\rho_\bfff$ arising from the $p$-adic \'{e}lale cohomology of the $p$-tower of modular curves of tame level $N_f$ as in the work of Ohta \cite{ohta-etI,ohta-etII} (see also \cite[\S{7.2}]{KLZ2}, whose conventions we adopt). We shall exploit the following fact.

\begin{prop}\label{prop:CM-ind}
Let $\bff=\boldsymbol{\theta}_\xi(S)$ be a primitive CM Hida family associated to the finite order character $\xi$ as in \eqref{eq:def-CM}. If $\bar{\xi}^-\vert_{G_{K_\pp}}\neq 1$, where $G_{K_\pp}\subset G_K$ is a decomposition group at $\pp$, then
\[
V_{\bff}\simeq{\rm Ind}_K^\bQ(T_\xi\otimes\Psi_{T})
\]
as $\mathscr{O}\dBr{S}[G_\bQ]$-modules, where $T=\mathbf{u}^{-1}(1+S)-1$.
\end{prop}

\begin{proof}
By \eqref{eq:residual-CM}, the condition $\bar{\xi}^-\vert_{G_{K_\pp}}\neq 1$ implies that the residual representation $\bar{\rho}_\bff$ is absolutely irreducible and $p$-distinguished. Thus, the result follows from a direct extension of the argument in \cite[Cor.~5.2.5]{LLZ-K} (see \cite[\S{3.2.3}]{BL-ord} for details).
\end{proof}

Suppose now that $(\varphi,\bfg,\bfh)$ is a CM triple as in $\S\ref{subsec:factor-L}$, with $\varphi=\theta(\lambda_0)\in S_2(\Gamma_1(N_\varphi))$ the newform associated to the specialisation of $\bff=\boldsymbol{\theta}_{\xi_0}(S_0)$ at $Q_0\in\mathfrak{X}_{\mathscr{O}\dBr{S_0}}^+$, and $(\bfg,\bfh)=(\boldsymbol{\theta}_{\lambda_1}(S_1),\boldsymbol{\theta}_{\lambda_2}(S_2))$ primitive CM Hida families associated to the ray class characters $\lambda_1,\lambda_2$. Suppose
\begin{equation}\label{eq:lambdai-dist}
\lambda_i(\pp)\not\equiv\lambda_i(\ppbar)\pmod{p}\quad\textrm{for $i=0,1,2$.}
\end{equation}
Letting $\Vdagsp$ be the image of $\Vdag$ under the specialisation map $(\ref{eq:sp})$ at $Q_0$, from Proposition~\ref{prop:CM-ind} we obtain
\[
\Vdagsp\simeq\Ind_K^\Q(T_{\lm_0})\otimes\Ind_K^\Q(T_{\lm_1}\otimes\Psi_{T_1})\otimes\Ind_K^\Q(T_{\lm_2}\otimes\Psi_{T_2})\otimes\mathcal{X}^{-1},
\]
where $\mathcal{X}=(\Psi_{T_1}^{1/2}\Psi_{T_2}^{1/2}\circ\mathscr{V}):G_\Q\rightarrow\Z_p\dBr{S_1,S_2}^\times$ with $T_i=\mathbf{u}^{-1}(1+S_i)-1$, and so an immediate computation shows that  
\begin{equation}\label{eq:dec-V}
\begin{aligned}
\Vdagsp
&\simeq\Ind_K^\Q(T_{\lm_0\lm_1\lm_2}\otimes\Psi_{W_1}^{1-\cc})\oplus\Ind_K^\Q(T_{\lm_0\lm_1^{\cc}\lm_2^{\cc}}\otimes\Psi_{W_1}^{\cc-1})\\
&\quad\oplus
\Ind_K^\Q(T_{\lm_0\lm_1\lm_2^{\cc}}\otimes\Psi_{W_2}^{1-\cc})\oplus
\Ind_K^\Q(T_{\lm_0\lm_1^{\cc}\lm_2}\otimes\Psi_{W_2}^{\cc-1}),
\end{aligned}
\end{equation}
where $W_1,W_2$ are as in Proposition~\ref{prop:factor-L}. 
Therefore, 
\begin{equation}\label{eq:shapiro}
\begin{aligned}
\rH^1(\Q,\Vdagsp)&\simeq\rH^1(K,T_{\lm_0\lm_1\lm_2}\otimes\Psi_{W_1}^{1-\cc})\oplus\rH^1(K,T_{\lm_0\lm_1^\cc\lm_2^\cc}\otimes\Psi_{W_1}^{\cc-1})\\
&\quad\oplus\rH^1(K,T_{\lm_0\lm_1\lm_2^\cc}\otimes\Psi_{W_2}^{1-\cc})\oplus\rH^1(K,T_{\lm_0\lm_1^\cc\lm_2}\otimes\Psi_{W_2}^{\cc-1})
\end{aligned}
\end{equation}
by Shapiro's lemma. 

\begin{prop}\label{prop:factor-S}
Under (\ref{eq:shapiro}), the balanced Selmer group 
decomposes as
\begin{align*}
{\rm Sel}^{\rm bal}(\Q,\Vdagsp)&\simeq{\rm Sel}_{\emptyset,0}(K,T_{\lm_0\lm_1\lm_2}\otimes\Psi_{W_1}^{1-\cc})\oplus{\rm Sel}_{0,\emptyset}(K,T_{\lm_0\lm_1^\cc\lm_2^\cc}\otimes\Psi_{W_1}^{\cc-1})\\
&\quad\oplus{\rm Sel}_{\emptyset,0}(K,T_{\lm_0\lm_1\lm_2^\cc}\otimes\Psi_{W_2}^{1-\cc})\oplus{\rm Sel}_{\emptyset,0}(K,T_{\lm_0\lm_1^\cc\lm_2}\otimes\Psi_{W_2}^{\cc-1}),
\end{align*}
and the $\varphi$-unbalanced Selmer group 
decomposes as
\begin{align*}
{\rm Sel}^{\unb}(\Q,\Vdagsp)&\simeq{\rm Sel}_{\emptyset,0}(K,T_{\lm_0\lm_1\lm_2}\otimes\Psi_{W_1}^{1-\cc})\oplus{\rm Sel}_{\emptyset,0}(K,T_{\lm_0\lm_1^\cc\lm_2^\cc}\otimes\Psi_{W_1}^{\cc-1})\\
&\quad\oplus{\rm Sel}_{\emptyset,0}(K,T_{\lm_0\lm_1\lm_2^\cc}\otimes\Psi_{W_2}^{1-\cc})\oplus{\rm Sel}_{\emptyset,0}(K,T_{\lm_0\lm_1^\cc\lm_2}\otimes\Psi_{W_2}^{\cc-1}).
\end{align*}
\end{prop}

The proof will follow easily from the following.

\begin{lem}\label{lem:shapiro}
Under $(\ref{eq:shapiro})$, the Selmer group ${\rm Sel}^{\rm bal}(\Q,\Vdagsp)$ corresponds to the submodule of 
\[
\rH^1(K,(T_{\lm_0\lm_1\lm_2}\otimes\Psi_{W_1}^{1-\cc})\oplus  (T_{\lm_0\lm_1^\cc\lm_2^\cc}\otimes\Psi_{W_1}^{\cc-1})\oplus (T_{\lm_0\lm_1\lm_2^\cc}\otimes\Psi_{W_2}^{1-\cc})\oplus (T_{\lm_0\lm_1^\cc\lm_2}\otimes\Psi_{W_2}^{\cc-1}))
\] 
consisting of unramified-outside-$p$ classes $x$ with ${\rm res}_v(x)$ belonging to
\[
\begin{cases}
\rH^1(K_\pp,(T_{\lm_0\lm_1\lm_2}\otimes\Psi_{W_1}^{1-\cc})\oplus(T_{\lm_0\lm_1\lm_2^\cc}\otimes\Psi_{W_2}^{1-\cc})\oplus (T_{\lm_0\lm_1^\cc\lm_2}\otimes\Psi_{W_2}^{\cc-1}))&\textrm{if $v=\pp$,}\\[0.2em]
\rH^1(K_{\ppbar},T_{\lm_0\lm_1^\cc\lm_2^\cc}\otimes\Psi_{W_1}^{\cc-1})&\textrm{if $v=\ppbar$,}
\end{cases}
\]
and the Selmer group ${\rm Sel}^{\unb}(\Q,\Vdagsp)$ similarly corresponds to the submodule consisting of unramified-outside-$p$ classes $x$ with ${\rm res}_{\ppbar}(x)=0$ (and no condition at $\pp$).
\end{lem}

\begin{proof}
Using (\ref{eq:dec-V}) we see that the balanced local condition is given by
\begin{align*}
\mathscr{F}^{\rm bal}_p(\Vdagsp)
&=(T_{\lm_0\lm_1\lm_2}\otimes\Psi_{W_1}^{1-\cc})\oplus (T_{\lm_0^\cc\lm_1\lm_2}\otimes\Psi_{W_1}^{1-\cc})\\
&\quad\oplus (T_{\lm_0\lm_1\lm_2^\cc}\otimes\Psi_{W_2}^{1-\cc})\oplus(T_{\lm_0\lm_1^\cc\lm_2}\otimes\Psi_{W_2}^{\cc-1}),
\end{align*}
from where we obtain
\begin{equation}\label{eq:bal-shapiro}
\begin{aligned}
\mathscr{F}^{\rm bal}_\pp(\Vdagsp)&=(T_{\lm_0\lm_1\lm_2}\otimes\Psi_{W_1}^{1-\cc})\oplus (T_{\lm_0\lm_1\lm_2^\cc}\otimes\Psi_{W_2}^{1-\cc})\oplus (T_{\lm_0\lm_1^\cc\lm_2}\otimes\Psi_{W_2}^{\cc-1}),\\
\mathscr{F}^{\rm bal}_{\ppbar}(\Vdagsp)&=T_{\lm_0\lm_1^\cc\lm_2^\cc}\otimes\Psi_{W_1}^{\cc-1},
\end{aligned}
\end{equation}
yielding the stated descriptions of ${\rm Sel}^{\rm bal}(K,\Vdagsp)$. Similarly, 
we see that the $\varphi$-unbalanced local condition is given by
\begin{align*}
\mathscr{F}_p^{\unb}(\Vdagsp)
&=(T_{\lm_0\lm_1\lm_2}\otimes\Psi_{W_1}^{1-\cc})\oplus (T_{\lm_0\lm_1^\cc\lm_2^\cc}\otimes\Psi_{W_1}^{\cc-1})\\
&\quad\oplus (T_{\lm_0\lm_1\lm_2^\cc}\otimes\Psi_{W_2}^{1-\cc})\oplus(T_{\lm_0\lm_1^\cc\lm_2}\otimes\Psi_{W_2}^{\cc-1}),
\end{align*}
and this immediately yields the stated description of ${\rm Sel}^{\unb}(\Q,\Vdagsp)$.
\end{proof}


\begin{proof}[Proof of Proposition~\ref{prop:factor-S}]
Put 
\begin{equation}\label{eq:til-V}
\begin{aligned}
\widetilde{\mathbb{V}}_{\varphi\bfg\bfh}^\dagger&:=(T_{\lm_0\lm_1\lm_2}\otimes\Psi_{W_1}^{1-\cc})\oplus (T_{\lm_0\lm_1^\cc\lm_2^\cc}\otimes\Psi_{W_1}^{\cc-1})\\
&\quad\oplus (T_{\lm_0\lm_1\lm_2^\cc}\otimes\Psi_{W_2}^{1-\cc})\oplus(T_{\lm_0\lm_1^\cc\lm_2}\otimes\Psi_{W_2}^{\cc-1}),
\end{aligned}
\end{equation} 
so 
we have
\begin{equation}\label{eq:shapiro-h1}
\rH^1(\Q,\Vdagsp)\simeq\rH^1(K,\widetilde{\mathbb{V}}_{\varphi\bfg\bfh}^\dagger).
\end{equation}

Denoting by ${\rm Sel}^{\{p\}}(K,\widetilde{\mathbb{V}}_{\varphi\bfg\bfh}^\dagger)$ the submodule of $\rH^1(K,\widetilde{\mathbb{V}}_{\varphi\bfg\bfh}^\dagger)$ consisting of unramified-outside-$p$ classes, it follows from Lemma~\ref{lem:shapiro} that under the above isomorphism the balanced Selmer group ${\rm Sel}^{\rm bal}(\Q,\Vdagsp)$ corresponds to the kernel of the restriction map from ${\rm Sel}^{\{p\}}(K,\widetilde{\mathbb{V}}_{\varphi\bfg\bfh}^\dagger)$ to
\begin{align*}
&\frac{\rH^1(K_\pp,\widetilde{\mathbb{V}}_{\varphi\bfg\bfh}^\dagger)}{\rH^1(K_\pp,(T_{\lm_0\lm_1\lm_2}\otimes\Psi_{W_1}^{1-\cc})\oplus(T_{\lm_0\lm_1\lm_2^\cc}\otimes\Psi_{W_2}^{1-\cc})\oplus( T_{\lm_0\lm_1^\cc\lm_2}\otimes\Psi_{W_2}^{\cc-1}))}\times\frac{\rH^1(K_{\ppbar},\widetilde{\mathbb{V}}_{\varphi\bfg\bfh}^\dagger)}{\rH^1(K_{\ppbar},T_{\lm_0\lm_1^\cc\lm_2^\cc}\otimes\Psi_{W_1}^{\cc-1})},
\end{align*}
and this kernel is isomorphic to
\begin{align*}
{\rm Sel}_{\emptyset,0}(K,T_{\lm_0\lm_1\lm_2}\otimes\Psi_{W_1}^{1-\cc})\oplus{\rm Sel}_{\emptyset,0}(K,T_{\lm_0\lm_1\lm_2^\cc}\otimes\Psi_{W_2}^{1-\cc})&\oplus{\rm Sel}_{\emptyset,0}(K,T_{\lm_0\lm_1^\cc\lm_2}\otimes\Psi_{W_2}^{\cc-1})\\
&\quad\times{\rm Sel}_{0,\emptyset}(K,T_{\lm_0\lm_1^\cc\lm_2^\cc}\otimes\Psi_{W_1}^{\cc-1}).
\end{align*}
This shows the result for ${\rm Sel}^{\rm bal}(K,\Vdagsp)$, and the 
case of ${\rm Sel}^{\varphi}(K,\Vdagsp)$ follows from Lemma~\ref{lem:shapiro} in the same manner.
\end{proof}

As a consequence we also obtain the following decomposition for the Selmer groups with coefficients in $\Adagsp={\rm Hom}_{\Z_p}(\Vdagsp,\mu_{p^\infty})$, mirroring in the case of ${\rm Sel}^{\unb}(K,\Adagsp)$ the factorisation of $p$-adic $L$-functions in Proposition~\ref{prop:factor-L}.

\begin{cor}\label{cor:factor-S}
The balanced Selmer group ${\rm Sel}^{\rm bal}(\Q,\Adagsp)$ decomposes as
\begin{align*}
{\rm Sel}^{\rm bal}(\Q,\Adagsp)&\simeq{\rm Sel}_{0,\emptyset}(K,A_{\lm_0^\cc\lm_1^\cc\lm_2^\cc}\otimes\Psi_{W_1}^{\cc-1})\oplus{\rm Sel}_{\emptyset,0}(K,A_{\lm_0^\cc\lm_1\lm_2}\otimes\Psi_{W_1}^{1-\cc})\\
&\quad\oplus{\rm Sel}_{0,\emptyset}(K,A_{\lm_0^\cc\lm_1^\cc\lm_2}\otimes\Psi_{W_2}^{\cc-1})\oplus{\rm Sel}_{0,\emptyset}(K,A_{\lm_0^\cc\lm_1\lm_2^\cc}\otimes\Psi_{W_2}^{1-\cc}),
\end{align*}
and the $\varphi$-unbalanced Selmer group ${\rm Sel}^{\unb}(\Q,\Adagsp)$ decomposes as
\begin{align*}
{\rm Sel}^{\unb}(\Q,\Adagsp)&\simeq{\rm Sel}_{0,\emptyset}(K,A_{\lm_0^\cc\lm_1^\cc\lm_2^\cc}\otimes\Psi_{W_1}^{\cc-1})\oplus{\rm Sel}_{0,\emptyset}(K,A_{\lm_0^\cc\lm_1\lm_2}\otimes\Psi_{W_1}^{1-\cc})\\
&\quad\oplus{\rm Sel}_{0,\emptyset}(K,A_{\lm_0^\cc\lm_1^\cc\lm_2}\otimes\Psi_{W_2}^{\cc-1})\oplus{\rm Sel}_{0,\emptyset}(K,A_{\lm_0^\cc\lm_1\lm_2^\cc}\otimes\Psi_{W_2}^{1-\cc}).
\end{align*}
\end{cor}

\begin{proof}
This is immediate from Proposition~\ref{prop:factor-S} and local Tate duality.
\end{proof}

\begin{rem}\label{rem:reversed}
Note that the decompositions of ${\rm Sel}^{\rm bal}(\Q,\Adagsp)$ and ${\rm Sel}^{\unb}(\Q,\Adagsp)$ in Corollary~\ref{cor:factor-S} only differ in their second direct summand: the term ${\rm Sel}_{\emptyset,0}(K,A_{\lm_0^\cc\lm_1\lm_2}\otimes\Psi_{W_1}^{1-\cc})$ in the former is replaced by ${\rm Sel}_{0,\emptyset}(K,A_{\lm_0^\cc\lm_1\lm_2}\otimes\Psi_{W_1}^{1-\cc})$ in the latter (i.e., the local conditions at the primes above $p$ are reversed). From the description in $\S\ref{subsec:Sel-char}$, it follows  that ${\rm Sel}_{0,\emptyset}(K,A_{\lm_0^\cc\lm_1\lm_2}\otimes\Psi_{W_1}^{1-\cc})$ corresponds to a Bloch--Kato Selmer group for $\lambda_0^\cc\lambda_1\lambda_2$ over the anticyclotomic $\Z_p$-extension of $K$. 
\end{rem}

%


\section{Diagonal cycle main conjecture}\label{sec:IMC}

In this section we give the proof of a two-variable variant of the diagonal cycle main conjecture formulated in \cite{ACR} specialised to the CM case.  

\subsection{Big diagonal classes}\label{subsec:big-diag}

Let $(\varphi,\bfg,\bfh)$ 
be a CM triple as in $\S\ref{subsec:factor-L}$, with $\varphi=\theta(\lm_0)\in S_2(\Gamma_1(N_\varphi))$ the 
newform corresponding to the specialisation of $\bff=\boldsymbol{\theta}_{\xi_0}(S_0)$ at $Q_0\in\mathfrak{X}^+_{\mathscr{O}\dBr{S_0}}$ and $(\bfg,\bfh)=(\boldsymbol{\theta}_{\lambda_1}(S_1),\boldsymbol{\theta}_{\lambda_2}(S_2))$ primitive CM Hida families associated to the ray class characters $\lambda_1,\lambda_2$. 

Denote by 
\begin{equation}\label{eq:diag}
\kappa(\varphi,\bfg,\bfh)\in\rH^1(\Q,\Vdagsp)
\end{equation}
the specialisation at $Q_0$ of the big diagonal class $\kappa(\bff,\bfg,\bfh)$ obtained from the construction in  \cite[\S{8.1}]{BSV}. 
%
%
More precisely, letting $N={\rm lcm}(N_\varphi,N_g,N_h)$, the construction in \emph{loc.\,cit.} produces a big diagonal class in the cohomology of a representation $\Vdagsp(N)$ non-canonically isomorphic to finitely many copies of $\Vdagsp$; our class (\ref{eq:diag}) is obtained by taking its image under the map $\Vdagsp(N)\rightarrow\Vdagsp$ associated to the level-$N$ test vectors for $(\varphi,\bfg,\bfh)$ furnished by \cite[Thm.~A]{hsieh-triple}.

Assume hypothesis \eqref{eq:lambdai-dist}. It follows from \cite[Cor.\,8.2]{BSV} that $\kappa(\varphi,\bfg,\bfh)$ lands in 
${\rm Sel}^{\rm bal}(\bQ,\Vdagsp)$, and therefore viewing this class in $\rH^1(K,\widetilde{\mathbb{V}}_{\varphi\bfg\bfh}^\dag)$ via \eqref{eq:shapiro-h1}, by Proposition~\ref{prop:factor-S} this implies that
\[
{\rm res}_{\ppbar}(\kappa(\varphi,\bfg,\bfh))\in\rH^1(K_{\ppbar},T_{\lm_0\lm_1^\cc\lm_2^\cc}\otimes\Psi_{W_1}^{\cc-1}),
\]
where $W_1=\mathbf{u}^{-1}(1+S_1)^{1/2}(1+S_2)^{1/2}-1$. 

For any $\mathscr{O}\dBr{S_1,S_2}$-module $M$ and integers $k_1,k_2\geq 1$  of the same parity, denote by $M_{k_1,k_2}$ the specialisation of $M$ at $S_i=\mathbf{u}^{k_i}-1$. Let $\Phi$ be the field of fractions of $\mathscr{O}$.  Then 
one easily checks that there are isomorphisms
\begin{equation}\label{eq:BK}
\begin{cases}
{\rm log}_{\ppbar}:\rH^1(K_{\ppbar},T_{\lm_0\lm_1^\cc\lm_2^\cc}\otimes\Psi_{W_1}^{\cc-1})_{k_1,k_2}\otimes\Q_p\rightarrow \Phi &\textrm{if $k_1+k_2\geq 3$},\\[0.2em]
{\rm exp}_{\ppbar}^*:\rH^1(K_{\ppbar},T_{\lm_0\lm_1^\cc\lm_2^\cc}\otimes\Psi_{W_1}^{\cc-1})_{k_1,k_2}\otimes\Q_p\rightarrow \Phi &\textrm{if $k_1+k_2=2$},
\end{cases}
\end{equation}
given by the Bloch--Kato logarithm and dual exponential maps, respectively.

\begin{thm}[Explicit reciprocity law]
\label{thm:ERL}
Let $(\varphi,\bfg,\bfh)=(\theta(\lm_0),\boldsymbol{\theta}_{\lm_1}(S_1)\boldsymbol{\theta}_{\lm_2}(S_2))$ be a CM triple as above, with $\mathcal{R}_{\varphi\bfg\bfh}=\mathscr{O}\dBr{S_1,S_2}$. 
There is an injective $\mathcal{R}_{\varphi\bfg\bfh}$-module homomorphism
\[
{\rm Log}^{\unb}:\rH^1(K_{\ppbar},T_{\lm_0\lm_1^\cc\lm_2^\cc}\otimes\Psi_{W_1}^{\cc-1})\rightarrow\mathcal{R}_{\varphi\bfg\bfh}
\]
with pseudo-null cokernel satisfying for any $\mathfrak{Z}\in\rH^1(K_{\ppbar},T_{\lm_0\lm_1^\cc\lm_2^\cc}\otimes\Psi_{W_1}^{\cc-1})$ the interpolation property
\[
{\rm Log}^{\unb}(\mathfrak{Z})_{k_1,k_2}=c_{k_1,k_2}\times\begin{cases}
{\rm log}_{\ppbar}(\mathfrak{Z}_{k_1,k_2})&\textrm{if $k_1+k_2\geq 3$,}\\[0.2em]
{\rm exp}^*_{\ppbar}(\mathfrak{Z}_{k_1,k_2})&\textrm{if $k_1+k_2=2$,}
\end{cases}
\]
where $c_{k_1,k_2}$ is an explicit nonzero constant, and such that
\[
{\rm Log}^{\unb}\bigl({\rm res}_{\overline{\pp}}(\kappa(\varphi,\bfg,\bfh))\bigr)=\mathscr{L}_p^\varphi(\varphi,\bfg,\bfh).
\]
\end{thm}

\begin{proof}
Let $\mathscr{F}_p^{3}(\Vdag)$ be the rank one subspace of $\mathscr{F}_p^{2}(\Vdag)=\mathscr{F}_p^{\rm bal}(\Vdag)$ given by
\[
\mathscr{F}_p^{3}(\Vdag)=V_\bff^+\hat\otimes_{\mathscr{O}} V_{\bfg}^+\hat\otimes_{\mathscr{O}} V_{\bfh}^+\otimes\mathcal{X}^{-1}.
\] 
Then one immediately finds that $\mathscr{F}^2_p(\Vdag)/\mathscr{F}_p^3(\Vdag)$ naturally contains the representation
\begin{equation}\label{eq:gr2}
\mathbf{V}_{\bff}^{\bfg\bfh}:=(V_\bff/V_\bff^+)\hat\otimes_{\mathscr{O}} V_{\bfg}^+\hat\otimes_{\mathscr{O}} V_{\bfh}^+\otimes\mathcal{X}^{-1}
\end{equation}
as a direct summand. Letting $\mathscr{F}_p^{3}(\Vdagsp)$ be the specialisation of $\mathscr{F}_p^{3}(\mathbf{V}^\dagger)$, and likewise for $\mathbb{V}_\varphi^{\bfg\bfh}$, in terms of the description given in the proof of Lemma~\ref{lem:shapiro}, we find that 
\[
\mathscr{F}_\pp^3(\Vdagsp)=T_{\lm_0\lm_1\lm_2}\otimes\Psi_{W_1}^{1-\cc},\quad\quad\mathscr{F}_{\ppbar}^3(\Vdagsp)=0,
\]
and that 
\[
\mathbb{V}_{\varphi}^{\bfg\bfh}\simeq\mathscr{F}_{\ppbar}^{\rm bal}(\Vdagsp)\simeq T_{\lm_0\lm_1^\cc\lm_2^\cc}\otimes\Psi_{W_1}^{\cc-1}.
\]
The construction of the map ${\rm Log}^{\unb}$ thus follows from specialising at $Q_0$ the three-variable $p$-adic regulator map  of \cite[\S{7.1}]{BSV} similarly as in the proof of \cite[Prop.~7.3]{ACR}. (Alternatively, a construction of ${\rm Log}^\varphi$ can be obtained from a generalised Coleman power series map \cite[App.]{kobayashi-PRtwist-I} for an appropriate Lubin--Tate extension of $\Q_p$, cf. \cite[Thm.~3.4]{cas-hsieh-ord}.) 
The stated explicit reciprocity law is a consequence of \cite[Thm.~A]{BSV}. 
\end{proof}

As explained in \cite[\S{7.3}]{ACR}, the following result can be seen as the equivalence between two formulations (`without $p$-adic $L$-functions' and `with $p$-adic $L$-functions', respectively) of the Iwasawa--Greenberg main conjecture \cite{Greenberg55} for $\Vdagsp$.

\begin{prop}\label{prop:equiv}
The following statements (1)-(2) are equivalent:
\begin{enumerate}
\item{} $\kappa(\varphi,\bfg,\bfh)$ is not $\mathcal{R}_{\varphi\bfg\bfh}$-torsion,
\[
{\rm rank}_{\mathcal{R}_{\varphi\bfg\bfh}}\bigl({\rm Sel}^{\rm bal}(\Q,\Vdagsp)\bigr)={\rm rank}_{\mathcal{R}_{\varphi\bfg\bfh}}\bigl(X^{\rm bal}(\Q,\Adagsp)\bigr)=1,
\]
and the following equality holds in $\mathcal{R}_{\varphi\bfg\bfh}\otimes\Q_p$:
\[
{\rm char}_{\mathcal{R}_{\varphi\bfg\bfh}}\bigl(X^{\rm bal}(\Q,\Adagsp)_{\rm tors}\bigr)={\rm char}_{\mathcal{R}_{\varphi\bfg\bfh}}\biggl(\frac{{\rm Sel}^{\rm bal}(\Q,\Vdagsp)}{(\kappa(\varphi,\bfg,\bfh))}\biggr)^2.
\]
\item{} $\mathscr{L}_p^\varphi(\varphi,\bfg,\bfh)$ is nonzero, 
\[
{\rm rank}_{\mathcal{R}_{\varphi\bfg\bfh}}\bigl({\rm}{\rm Sel}^{\unb}(\Q,\Vdagsp)={\rm rank}_{\mathcal{R}_{\varphi\bfg\bfh}}\bigl(X^{\unb}(\Q,\Adagsp)\bigr)=0, 
\]
and the following equality holds in $\mathcal{R}_{\varphi\bfg\bfh}\otimes\Q_p$:
\[
{\rm char}_{\mathcal{R}_{\varphi\bfg\bfh}}\bigl(X^{\unb}(\Q,\Adagsp)\bigr)=\bigl(\mathscr{L}_p^\varphi(\varphi,\bfg,\bfh)\bigr)^2.
\]
\end{enumerate}
\end{prop}

\begin{proof}
This can be deduced from Theorem~\ref{thm:ERL} and global duality in the same way as \cite[Thm.~7.15]{ACR}. 
See \cite{lai-PhD} for the details in the stated level of generality. 
%
\end{proof}

\subsection{Proof of Theorem~\ref{thmintro:A}}

We can now conclude the proof of our main first main result towards the diagonal cycle main conjecture (Conjecture~\ref{conj:DCMC}). 

Recall that we let $K$ be an imaginary quadratic field in which the prime $p\geq 5$ satisfies \eqref{eq:spl} and (for simplicity) $p\nmid h_K$; and 
\[
(\varphi,\bfg,\bfh)=(\theta(\lambda_0),\boldsymbol{\theta}(\lambda_1),\boldsymbol{\theta}(\lambda_1))
\]
is a CM triple as in $\S\ref{subsec:factor-L}$ associated to a triple of Hecke characters $(\lambda_0,\lambda_1,\lambda_2)$ satisfying \eqref{eq:sd-lambda} and of conductors coprime to $p$.
%
%
%

\begin{thm}\label{thm:A}
Suppose that: 
\begin{itemize}
\item[(i)] $\lambda_i(\pp)\not\equiv\lambda_i(\ppbar)\pmod{p}$ for $i=0,1,2$.
\item[(ii)] ${\rm sign}(\lm_0\lm_1\lm_2)={\rm sign}(\lm_0\lm_1^\cc\lm_2^\cc)={\rm sign}(\lm_0\lm_1\lm_2^\cc)={\rm sign}(\lm_0\lm_1^\cc\lm_2)=+1$. 
\end{itemize}
Then the class $\kappa(\varphi,\bfg,\bfh)$ is not $\mathcal{R}_{\varphi\bfg\bfh}$-torsion, the modules ${\rm Sel}^{\rm bal}(\Q,\mathbb{V}_{\varphi\bfg\bfh}^\dagger)$ and $X^{\rm bal}(\bQ,\mathbb{A}_{\varphi\bfg\bfh}^\dagger)$ both have $\mathcal{R}_{\varphi\bfg\bfh}$-rank one, and the following equality holds in $\mathcal{R}_{\varphi\bfg\bfh}\otimes\Q_p$: 
\[
{\rm char}_{\mathcal{R}_{\varphi\bfg\bfh}}\bigl(X^{\rm bal}(\Q,\mathbb{A}_{\varphi\bfg\bfh}^\dagger)_{\rm tors}\bigr)={\rm char}_{\mathcal{R}_{\varphi\bfg\bfh}}\biggl(\frac{{\rm Sel}^{\rm bal}(\Q,\mathbb{V}_{\varphi\bfg\bfh}^\dagger)}{\mathcal{R}_{\varphi\bfg\bfh}\cdot\kappa(\varphi,\bfg,\bfh)}\biggr)^2.
\]
In other words, Conjecture~\ref{conj:DCMC} specialised to the triple $(\varphi,\bfg,\bfh)$ holds.
\end{thm}

\begin{proof}
By Greenberg's nonvanishing results \cite{greenberg-BSD,greenberg-critical} (see \cite[Cor.\,2.1.5]{AHsplit} and \cite[Prop.\,2.3]{Arnold}), hypothesis (ii) implies that the four anticyclotomic Katz $p$-adic $L$-functions appearing in the factorisation of Proposition~\ref{prop:factor-L} are nonzero, and so we deduce that $\mathscr{L}_p^\varphi(\varphi,\bfg,\bfh)$ is also nonzero. 
Moreover,  Theorem\,\ref{thm:AH} implies that each of the direct summands in the decomposition of $X^\varphi(\bQ,\Adagsp)$ given in Proposition\,\ref{prop:factor-S} is torsion, with characteristic ideal generated by a corresponding anticyclotomic Katz $p$-adic $L$-function. Therefore $X^\varphi(\bQ,\Adagsp)$ is $\mathcal{R}_{\varphi\bfg\bfh}$-torsion, with 
\[
{\rm char}_{\mathcal{R}_{\varphi\bfg\bfh}}\bigl(X^\varphi(\bQ,\Adagsp)
\bigr)=\bigl(\mathscr{L}_p^\varphi(\varphi,\bfg,\bfh)\bigr)^2
\]
in $\mathcal{R}_{\varphi\bfg\bfh}\otimes\Q_p$. Together with the equivalence of Proposition\,\ref{prop:equiv}, this concludes the proof. 
\end{proof}

\section{Generalised Kato classes}\label{sec:GKC}

In this section we define generalised Kato classes attached to CM elliptic curves $E/\bQ$, following a slight modification of the construction due to Darmon--Rotger \cite{DR2.5}. Then we prove our main result on the nonvanishing of these classes in situations of rank $2$.

\subsection{Construction of the classes}\label{subsec:setting}

Let $E/\Q$ be an elliptic curve with CM by the maximal order $\cO_K$. Let $p\geq 5$ be a prime of good \emph{ordinary} reduction for $E$; in particular, (\ref{eq:spl}) holds.  Let $\psi_E$ be the Grossencharacter of $K$ associated to $E$ by the theory of complex multiplication, so we have 
\[
L(E,s)=L(\psi_E,s). 
\] 
Note that $L(\psi_E,s)=L(\psi_E^\cc,s)=L(\psi_E^{-1},s-1)$, using the relations $\psi_E^\cc=\overline{\psi_E}$ and $\psi_E\overline{\psi_E}=\mathbf{N}$ for the second equality.


Let $\ell_1, \ell_2$ be distinct primes split in $K$ with $(\ell_1 \ell_2,pN_E)=1$, where $N_E$ is the conductor of $E$. For $i=1,2$, let $\phi_i$  be a ring class character of conductor $\ell_i^{m_i}\cO_K$ for some $m_i>0$. 

\begin{prop}
\label{prop:nonvanishing}
Suppose $E$ has root number $+1$. Then there exist infinitely many ring class characters $\phi_1,\phi_2$ as above with 
\[
\phi_1\vert_{G_{K_\pp}},\phi_2\vert_{G_{K_\pp}}\neq 1,
\] 
where $G_{K_\pp}\subset G_K$ is a decomposition group at $\pp$, and such that
\[
L(\psi_E^{-1}\phi_1^{-1},0)\cdot L(\psi_E^{-1}\phi_2^{-1},0)\cdot L(\psi_E^{-1}\phi_1^{-1}\phi_2^{-1},0)\neq 0.
\]
\end{prop}

\begin{proof}
Since the primes $\ell_1$ and $\ell_2$ both split in $K$, for any ring class characters $\phi_1,\phi_2$ as above the signs in the functional equations for $L(\psi_E^{-1}\phi_1^{-1},s), L(\psi_E^{-1}\phi_2^{-1},s), L(\psi_E^{-1}\phi_1^{-1}\phi_2^{-1},s)$ are the same as the sign of $L(E,s)$. Thus if $E$ has root number $+1$, the nonvanishing results of \cite{greenberg-critical} and \cite{rohrlich-ac} imply that, as $\phi_i$ vary in the set of ring class character of $\ell_i$-power conductor for $i=1,2$, only finitely many of the values $L(\psi_E^{-1}\phi_1^{-1},0)\cdot L(\psi_E^{-1}\phi_2^{-1},0)\cdot L(\psi_E^{-1}\phi_1^{-1}\phi_2^{-1},0)$ are zero, whence the result.
\end{proof}

We now fix a pair of ring class characters $\phi_1,\phi_2$ as above, and writing
\[
\phi_i=\lambda_i^{1-\cc} 
\]
with $\lambda_i$ a ray class character modulo $\ell^{m_i}\cO_K$ (see \cite[Lem.\,5.31]{HMI}, for example), we consider the CM triple 
\begin{equation}\label{eq:choice}
(\varphi,\bfg,\bfh)=(\theta(\psi_E\lm_1^{-\cc}\lm_2^{-\cc}),\boldsymbol{\theta}_{\lm_1}(S_1),\boldsymbol{\theta}_{\lm_2}(S_2)).
\end{equation}


%

The following definition 
might be seen as a twisted variant of the \emph{generalised Kato classes} introduced by Darmon--Rotger \cite{DR2.5} (whose construction in the setting conjecturally of relevance for rank two elliptic curves over $\Q$, namely the ``adjoint case''studied in [\emph{op.\,cit.}, \S{4.3}], would suggest taking $\varphi=\theta(\psi_E)$, rather than $\varphi$ as in \eqref{eq:choice}).

\begin{defn}[Generalised Kato class]
\label{def:GKC}
Denote by $\Vdags$ the restriction of 
$\Vdagsp$ to $S_1=S_2$, let $\kappa(\varphi,\underline{\bfg\bfh})\in\rH^1(\Q,\Vdags)$ be the corresponding projection of \eqref{eq:diag}, and put $S=S_1$. We define $\kappa_p$ to be the image of $\kappa(\varphi,\underline{\bfg\bfh})$ under the composition
\[
{\rm Sel}^{\rm bal}(\Q,\Vdags)\rightarrow{\rm Sel}_{0,\emptyset}(K,T_{\psi_E}\otimes\Psi_{S}^{\cc-1})\rightarrow{\rm Sel}_{0,\emptyset}(K,T_{\psi_E}),
\]
where the first arrow is given by the projection onto the second direct summand from the decomposition of ${\rm Sel}^{\rm bal}(\Q,\Vdags)$ from  Proposition~\ref{prop:factor-S} (with $\lambda_0:=\psi_E\lambda_1^{-\cc}\lambda_2^{-\cc}$), and the second arrow is induced by the multiplication by $S$ on $T_{\psi_E}\otimes\Psi_S^{\cc-1}$. 
\end{defn}


Using $\rH^1(K,T_{\psi_E})\otimes\Q_p\simeq\rH^1(\Q,V_pE)$, we shall also view $\kappa_p$ in the latter group.

\begin{lem}\label{lem:ERL}
Suppose $L(E,1)=0$. Then $\kappa_p\in{\rm Sel}(\Q,V_pE)$.
\end{lem}

\begin{proof}
Viewing $\kappa_p$ in ${\rm Sel}_{0,\emptyset}(K,T_{\psi_E})$, it suffices to show that ${\rm res}_{\ppbar}(\kappa_p)=0$. 
The assumption implies that $L(\psi_E^{-1},0)=0$, and so from Theorem\,\ref{thm:katz}, Proposition\,\ref{prop:factor-L}, and Theorem\,\ref{thm:ERL}, we obtain
\begin{align*}
L(E,1)=0\quad\Longrightarrow\quad\mathcal{L}_{\psi_E}^-(0)=0
\quad&\Longrightarrow\quad\mathscr{L}_p^\varphi(\varphi,\bfg,\bfh)(\mathbf{u}-1,\mathbf{u}-1)=0\\
\quad&\Longrightarrow\quad{\rm exp}^*_{\ppbar}(\kappa(\varphi,\bfg,\bfh)_{1,1})=0,
\end{align*}
where, as in Theorem\,\ref{thm:ERL}, $\kappa(\varphi,\bfg,\bfh)_{k_1,k_2}$ denotes the specialisation of $\kappa(\varphi,\bfg,\bfh)$ at $S_i=\mathbf{u}^{k_i}-1$. Since for $k_1+k_2=2$ the Bloch--Kato dual exponential \eqref{eq:BK} is injective, this gives the result.
\end{proof}

\subsection{A control theorem}

Let 
\[
\Lambda=\mathscr{O}\dBr{S}.
\] 
For any $\Lambda$-module $M$, we $M_{/S}$ for the cokernel of multiplication by $S$, i.e., $M_{/S}:=M/SM$. The following is a variant of Mazur's control theorem for the ``reversed'' Selmer groups of Remark~\ref{rem:reversed}.

\begin{prop}\label{prop:control}
There are natural isomorphisms
\begin{align*}
r^*:{\rm Sel}_{\emptyset,0}(K,A_{\psi_E^{\cc}})&\rightarrow{\rm Sel}_{\emptyset,0}(K,A_{\psi_E^{\cc}}\otimes\Psi_{S}^{1-\cc})[S],\\
r:{\rm Sel}_{0,\emptyset}(K,T_{\psi_E}\otimes\Psi_{S}^{\cc-1})_{/S}&\rightarrow{\rm Sel}_{0,\emptyset}(K,T_{\psi_E}),
\end{align*}
induced by multiplication by $S$.
\end{prop}

\begin{proof}
The map $r^*$ fits into the commutative diagram with exact rows
\begin{equation}\label{eq:diag-res}
\resizebox{\displaywidth}{!}{
\xymatrix{
0\ar[r]&{\rm Sel}_{\emptyset,0}(K,A_{\psi_E^{\cc}})\ar[r]\ar[d]^{r^*}&{\rm Sel}^{\{p\}}(K,A_{\psi_E^{\cc}})\ar[r]\ar[d]^{s^*}&\rH^1(K_{\ppbar},A_{\psi_E^{\cc}})\ar[d]^{t^*}\\
0\ar[r]&{\rm Sel}_{\emptyset,0}(K,A_{\psi_E^{\cc}}\otimes\Psi_{S}^{1-\cc})[S]\ar[r]&{\rm Sel}^{\{p\}}(K,A_{\psi_E^{\cc}}\otimes\Psi_{S}^{1-\cc})[S]\ar[r]&\rH^1(K_{\ppbar},A_{\psi_E^{\cc}}\otimes\Psi_{S}^{1-\cc})[S].\nonumber
}}
\end{equation}
It follows from Lubin--Tate theory that $K_{\ppbar}(E[\ppbar])$ has degree $p-1$ over $K_{\ppbar}$ (see e.g. \cite[Ch.~I]{deshalit}), and therefore $\rH^0(K_{\infty,\ppbar},E[\ppbar^\infty])=0$ since $\Gamma^-$ is pro-$p$. By inflation-restriction, it follows that the map $t^*$ is an isomorphism. 

The map $s^*$ fits into the commutative diagram with exact rows 
\begin{equation}\label{eq:diag-res}
\resizebox{\displaywidth}{!}{
\xymatrix{
0\ar[r]&{\rm Sel}^{\{p\}}(K,A_{\psi_E^{\cc}})\ar[r]\ar[d]^{s^*}&\rH^1(K_\Sigma/K,A_{\psi_E^{\cc}})\ar[r]\ar[d]^{u^*}&\bigoplus_{w\in\Sigma,w\nmid p}\rH^1(K_w,A_{\psi_E^{\cc}})\ar[d]^{v^*}\\
0\ar[r]&{\rm Sel}^{\{p\}}(K,A_{\psi_E^{\cc}}\otimes\Psi_{S}^{1-\cc})[S]\ar[r]&\rH^1(K_\Sigma/K,A_{\psi_E^{\cc}}\otimes\Psi_{S}^{1-\cc})[S]\ar[r]&\bigoplus_{w\in\Sigma,w\nmid p}\rH^1(K_w,A_{\psi_E^{\cc}}\otimes\Psi_{S}^{1-\cc})[S].\nonumber
}}
\end{equation}
The vanishing of $\rH^0(K_{\infty,\ppbar},E[\ppbar^\infty])$ implies that $u^*$ is an isomorphism, and as shown in the proof of \cite[Lem.\,IV.3.5]{deshalit} as a consequence of \cite[Cor.\,4.4]{mazur-18}, the map $v^*$ is injective. 
%

Therefore, by the Snake Lemma applied to the above two diagrams 
we conclude that ${\rm ker}(r^*)={\rm ker}(s^*)=0$ and that ${\rm coker}(r^*)={\rm coker}(s^*)={\rm ker}(v^*)=0$. This gives the result for $r^*$, and the case of $r$ is shown in the same manner.
\end{proof}



Finally, for the proof of our second main result we shall use the following relation 
between the rank of the ``reversed'' Selmer group for $\psi_E$ and that of its usual Selmer group.

\begin{lem}\label{lem:-1}
We have
\[
{\rm rank}_{\cO_{K,\pp}}{\rm Sel}_{0,\emptyset}(K,T_{\psi_E})=\begin{cases}
{\rm rank}_{\cO_{K,\pp}}{\rm Sel}(K,T_{\psi_E})-1 &\textrm{if ${\rm res}_p\neq 0$,}\\[0.2em]
{\rm rank}_{\cO_{K,\pp}}{\rm Sel}(K,T_{\psi_E})+1&\textrm{if ${\rm res}_p=0$,}
\end{cases}
\]
where ${\rm res}_p:{\rm Sel}(\Q,V_pE)\rightarrow E(\Q_p)\hat{\otimes}\Q_p$ the restriction map at $p$.
\end{lem} 

\begin{proof}
By global duality we have the exact sequence
\begin{equation}\label{eq:PT-bisi}
0\rightarrow{\rm Sel}(K,T_{\psi_E})\rightarrow{\rm Sel}_{\rm rel}(K,T_{\psi_E})\xrightarrow{\alpha}\prod_{v\vert p}\frac{\rH^1(K_v,T_{\psi_E})}{E(K_v)\otimes\cO_{K,\pp}}\xrightarrow{\beta^\vee}{\rm Sel}(K,A_{\psi_E^\cc})^\vee,
\end{equation}
where the last arrow corresponds (by Tate's local duality) to the Pontryagin dual of the restriction map 
\[
\beta:{\rm Sel}(K,A_{\psi_E^{\cc}})\rightarrow\prod_{v\vert p}E(K_v)\otimes(K_{\ppbar}/\cO_{K,\ppbar}).
\]
Note that the target of the map $\beta$ has $\cO_{K,\pp}$-rank one, and from the action of complex conjugation and the isomorphism ${\rm Sel}(K,A_{\psi_E})\simeq{\rm Sel}_{p^\infty}(E/\bQ)$ we see that ${\rm rank}_{\cO_{K,\pp}}{\rm im}(\beta)={\rm dim}_{\bQ_p}{\rm im}({\rm res}_p)$.

Suppose first that ${\rm res}_p\neq 0$, so by the above remarks the map $\beta$ has finite cokernel. By (\ref{eq:PT-bisi}), it follows that $\alpha$ has finite image, and therefore 
\[
{\rm rank}_{\cO_{K,\pp}}{\rm Sel}_{\rm rel}(K,T_{\psi_E})={\rm rank}_{\cO_{K,\pp}}{\rm Sel}(K,T_{\psi_E}).
\]
In particular, this implies that
\[
{\rm rank}_{\cO_{K,\pp}}{\rm Sel}_{0,\emptyset}(K,T_{\psi_E})=
{\rm rank}_{\cO_{K,\pp}}\ker\bigl({\rm res}_\pp:{\rm Sel}(K,T_{\psi_E})\rightarrow E(K_\pp)\otimes\cO_{K,\pp}\bigr),
\]
yielding the result in this case. On the other hand, if ${\rm res}_p=0$ then from \eqref{eq:PT-bisi} it follows that 
\[
{\rm rank}_{\cO_{K,\pp}}{\rm Sel}_{\rm rel}(K,T_{\psi_E})={\rm rank}_{\cO_{K,\pp}}{\rm Sel}(K,T_{\psi_E})+1.
\]
Since ${\rH}^1(K_\pp,T_{\psi_E})/E(K_\pp)\otimes\cO_{K,\pp}$ is torsion, the modules ${\rm Sel}_{\rm rel}(K,T_{\psi_E})$ and ${\rm Sel}_{0,\emptyset}(K,T_{\psi_E})$ have the same rank, so this concludes the proof.
\end{proof}

\subsection{Proof of Theorem~\ref{thmintro:B}}

We now prove our main result on the nonvanishing of generalised Kato classes $\kappa_p\in{\rm Sel}(\bQ,V_pE)$ in situations where ${\rm ord}_{s=1}L(E,s)\geq 2$. 

Recall that we let $E/\bQ$ be an elliptic curve with CM by an imaginary quadratic field $K$ in which the prime $p\geq 5$ splits as $(p)=\pp\ppbar$. Suppose $E$ has root number $+1$, and let $\kappa_p\in\rH^1(\Q,V_pE)$ be the generalised Kato class attached to a pair of ring class characters $\phi_1=\lambda_1^{1-\cc},\phi_2=\lambda_2^{1-\cc}$ as in Proposition~\ref{prop:nonvanishing}; then, by Lemma~\ref{lem:ERL}, $\kappa_p\in{\rm Sel}(\Q,V_pE)$ as long as $L(E,1)=0$. 

\begin{thm}
\label{thm:B}
%
Suppose $L(E,s)$ vanishes to positive even order at $s=1$. 
%
%
Then 
\[
\kappa_p\neq 0\quad\Longrightarrow\quad{\rm dim}_{\Q_p}{\rm Sel}(\Q,V_pE)=2.
\]
Conversely, if ${\rm dim}_{\Q_p}{\rm Sel}(\Q,V_pE)=2$ then $\kappa_p\neq 0$ if and only if the restriction map
\[
{\rm res}_p:{\rm Sel}(\bQ,V_pE)\rightarrow E(\bQ_p)\hat\otimes\bQ_p
\]
is nonzero. 
\end{thm}

\begin{proof}
Put $S=S_1$. Restricted to $S_1=S_2$, 
the factorisation in Proposition~\ref{prop:factor-L} (in which then $W_1=S$ and $W_2=0$) reads as the equality
\begin{equation}\label{eq:factor-L}
\mathscr{L}_p^\varphi(\varphi,\underline{\bfg\bfh})^2(S)=\cL^-_{\pp,\psi_E\phi_1\phi_2}(S)^\iota\cdot\cL^-_{\pp,\psi_E}(S)\cdot\cL_{\pp,\psi_E\phi_1}^-(0)\cdot\cL_{\pp,\psi_E\phi_2}^-(0)\nonumber
\end{equation}
up to a multiplication by a unit $u\in\cW[1/p]^\times$. On the other hand, the decomposition in Corollary~\ref{cor:factor-S} for the unbalanced Selmer group becomes
\begin{align*}
X^{\unb}(\Q,\Adags)&\simeq
X_{0,\emptyset}(K,A_{\psi_E^\cc\phi_1^\cc\phi_2^\cc}\otimes\Psi_{S}^{\cc-1})\oplus X_{0,\emptyset}(K,A_{\psi_E^\cc}\otimes\Psi_{S}^{1-\cc})\\
&\quad\oplus X_{0,\emptyset}(K,A_{\psi_E^\cc\phi_1^\cc}\otimes\Psi_{S}^{\cc-1})_{/S}\oplus X_{0,\emptyset}(K,A_{\psi_E^\cc\phi_2^\cc}\otimes\Psi_{S}^{1-\cc})_{/S}.
\end{align*}
Noting that the involution $\iota$ gives $X_{0,\emptyset}(K,A_{\psi_E^\cc\phi_1^\cc\phi_2^\cc}\otimes\Psi_{S}^{1-\cc})^\iota=X_{0,\emptyset}(K,A_{\psi_E^\cc\phi_1^\cc\phi_2^\cc}\otimes\Psi_{S}^{\cc-1})$, from Theorem~\ref{thm:AH} we deduce that $X^{\unb}(\Q,\Adags)$ is $\Lambda$-torsion with 
\[
{\rm char}_{\Lambda}\bigl(X^{\unb}(\Q,\Adags)\bigr)=\bigl(\mathscr{L}_p^\varphi(\varphi,\underline{\bfg\bfh})\bigr)^2
\]
as ideals in $\Lambda_\cW\otimes\Q_p$. By Proposition~\ref{prop:equiv}, it follows that 
\begin{equation}\label{eq:rank-bal}
{\rm rank}_\Lambda\bigl(X^{\rm bal}(\Q,\Adags)\big)={\rm rank}_\Lambda\bigl({\rm Sel}^{\rm bal}(\Q,\Vdags)\bigr)=1,
\end{equation}
and
\begin{equation}\label{eq:char-bal}
{\rm char}_{\Lambda}\bigl(X^{\rm bal}(\Q,\Adags)_{\rm tors}\bigr)={\rm char}_{\Lambda}\biggl(\frac{{\rm Sel}^{\rm bal}(\Q,\Vdags)}{\Lambda\cdot\kappa(\varphi,\underline{\bfg\bfh})}\biggr)^2
\end{equation}
as ideals in $\Lambda_\cW\otimes\Q_p$. Now, the implications
\begin{align*}
L(\psi_E^{-1}\phi_1^{-1}\phi_2^{-1},0)\neq 0\quad&\Longrightarrow\quad\cL_{\pp,\psi_E\phi_1\phi_2}^-(0)\neq 0\\
&\Longrightarrow\quad\left\vert X_{0,\emptyset}(K,A_{\psi_E^\cc\phi_1^\cc\phi_2^\cc}\otimes\Psi_{S}^{\cc-1})_{/S}\right\vert<\infty.
\end{align*}
follow from the interpolation property of $\mathcal{L}_{\pp,\psi_E\phi_1\phi_2}^-$ 
and the combination of Theorem~\ref{thm:AH} and Mazur's control theorem, 
respectively. The nonvanishing of $L(\psi_E^{-1}\phi_1^{-1},0)$ and $L(\psi_E^{-1}\phi_2^{-1},0)$ similarly implies
\[
\left\vert X_{0,\emptyset}(K,A_{\phi_E^\cc\phi_1^\cc}\otimes\Psi_{S}^{\cc-1})_{/S}\right\vert<\infty,\quad\left\vert X_{0,\emptyset}(K,A_{\phi_E^\cc\phi_2^\cc}\otimes\Psi_{S}^{\cc-1})_{/S}\right\vert<\infty,
\]
and so from  \eqref{eq:rank-bal} and the balanced Selmer group decomposition from Corollary~\ref{cor:factor-S}:  
\begin{align*}
X^{\rm bal}(\Q,\Adags)&\simeq
X_{0,\emptyset}(K,A_{\psi_E^\cc\phi_1^\cc\phi_2^\cc}\otimes\Psi_{S}^{\cc-1})\oplus X_{\emptyset,0}(K,A_{\psi_E^\cc}\otimes\Psi_{S}^{1-\cc})\\
&\quad\oplus X_{0,\emptyset}(K,A_{\psi_E^\cc\phi_1^\cc}\otimes\Psi_{S}^{\cc-1})_{/S}\oplus X_{0,\emptyset}(K,A_{\psi_E^\cc\phi_2^\cc}\otimes\Psi_{S}^{1-\cc})_{/S},
\end{align*}
we deduce that
\begin{equation}\label{eq:reversed-rank1}
{\rm rank}_\Lambda\bigl(X_{\emptyset,0}(K,A_{\psi_E^\cc}\otimes\Psi_{S}^{1-\cc})\bigr)=1.
\end{equation}
On the other hand, by the control theorem of Proposition~\ref{prop:control} the $\cO_{K,\pp}$-rank of the specialisation $X_{\emptyset,0}(\Q,A_{\psi_E^\cc}\otimes\Psi_{S}^{1-\cc})_{/S}$  satisfies
\begin{equation}\label{eq:coinv}
\begin{aligned}
{\rm rank}_{\cO_{K,\pp}}\bigl(X_{\emptyset,0}(K,A_{\psi_E^\cc}\otimes\Psi_{S}^{1-\cc})_{/S}\bigr)
&={\rm corank}_{\cO_{K,\pp}}\bigl({\rm Sel}_{\emptyset,0}(K,A_{\psi_E^\cc})\bigr)\\
&={\rm rank}_{\cO_{K,\pp}}\bigl({\rm Sel}_{\emptyset,0}(K,T_{\psi_E^{\cc}})\bigr)\\
&={\rm rank}_{\cO_{K,\pp}}\bigl({\rm Sel}_{0,\emptyset}(K,T_{\psi_E})\bigr),
\end{aligned}
\end{equation}
using the isomorphism ${\rm Sel}_{\emptyset,0}(K,T_{\psi_E^\cc})\simeq{\rm Sel}_{0,\emptyset}(K,T_{\psi_E})$ given by the action of complex conjugation for the last equality.

Denote by $\kappa_E(\varphi,\underline{\bfg\bfh})\in{\rm Sel}_{0,\emptyset}(K,T_{\psi_E}\otimes\Psi_{S}^{\cc-1})$ the projection of $\kappa(\varphi,\underline{\bfg\bfh})$ onto the second direct summand in the balanced Selmer groups decomposition from Proposition~\ref{prop:factor-S}:
\begin{align*}
{\rm Sel}^{\rm bal}(\Q,\Vdags)&\simeq
{\rm Sel}_{\emptyset,0}(K,T_{\psi_E\phi_1\phi_2}\otimes\Psi_{S}^{1-\cc})\oplus{\rm Sel}_{0,\emptyset}(K,T_{\psi_E}\otimes\Psi_{S}^{\cc-1})\\
&\quad
\oplus{\rm Sel}_{\emptyset,0}(K,T_{\psi_E\phi_1}\otimes\Psi_{S}^{1-\cc})_{/S}\oplus{\rm Sel}_{\emptyset,0}(K,T_{\psi_E\phi_2}\otimes\Psi_{S}^{\cc-1})_{/S}.
\end{align*}
Since our assumptions imply that $\psi_E\phi_1,\psi_E\phi_2,\psi_E\phi_1\phi_2$ have all root number $+1$, by Theorem\,\ref{thm:AH} we have
\[
\frac{{\rm Sel}^{\rm bal}(\Q,\Vdags)}{\Lambda\cdot\kappa(\varphi,\underline{\bfg\bfh})}\simeq\frac{{\rm Sel}_{0,\emptyset}(K,T_{\psi_E}\otimes\Psi_{S}^{\cc-1})}{\Lambda\cdot\kappa_E(\varphi,\underline{\bfg\bfh})}.
\]
Together with Theorem~\ref{thm:A} (note that the choice of $\phi_1,\phi_2$ in Proposition~\ref{prop:nonvanishing} guarantee that the associated primtitive CM Hida families $\bfg,\bfh$ satisfy the conditions in that result), this gives that ${\rm Sel}_{0,\emptyset}(K,T_{\psi_E}\otimes\Psi_{S}^{\cc-1})$ has $\Lambda$-rank one, and we have the following equality of ideals in $\Lambda_\cW\otimes\Q_p$:
\begin{equation}\label{eq:IMC-lambda}
{\rm char}_\Lambda\bigl(X_{\emptyset,0}(K,A_{\psi_E^\cc}\otimes\Psi_S^{1-\cc})_{\rm tors}\bigr)\cdot\mathcal{L}_{\pp,\psi_E^\cc\phi_1^\cc\phi_2^\cc}^-(S)^\iota={\rm char}_\Lambda\bigl(\mathfrak{Z}_E\bigr)^2,
\end{equation}
where $\mathfrak{Z}_E={\rm Sel}_{0,\emptyset}(K,T_{\psi_E}\otimes\Psi_{S}^{\cc-1})/\Lambda\cdot\kappa_E(\varphi,\underline{\bfg\bfh})$. From  (\ref{eq:reversed-rank1}), (\ref{eq:coinv}), and (\ref{eq:IMC-lambda}), we thus see that
\begin{align*}
{\rm rank}_{\cO_{K,\pp}}\bigl({\rm Sel}_{0,\emptyset}(K,T_{\psi_E})\bigr)&=1-{\rm ord}_S\bigl(\mathcal{L}_{\pp,\psi_E^\cc\phi_1^\cc\phi_2^\cc}^-(S)^\iota\bigr)+2\,{\rm rank}_{\cO_{K,\pp}}\bigl((\mathfrak{Z}_E)_{/S}\bigr)\\
&=1+2\,{\rm rank}_{\cO_{K,\pp}}\bigl((\mathfrak{Z}_E)_{/S}\bigr),
\end{align*}
using Theorem\,\ref{thm:katz} and the nonvanishing of $L(\psi_E^{-1}\phi_1^{-1}\phi_2^{-1},0)$ for the last equality. 

Since by construction the injection
\[
{\rm Sel}_{0,\emptyset}(K,T_{\psi_E}\otimes\Psi_{S}^{\cc-1})_{/S}\rightarrow{\rm Sel}_{0,\emptyset}(K,T_{\psi_E})
\]
of Proposition~\ref{prop:control} sends 
$\kappa_E(\varphi,\underline{\bfg\bfh})\;{\rm mod}\;S$ into $\kappa_p$, we conclude that
\begin{equation}\label{eq:str=1}
{\rm rank}_{\cO_{K,\pp}}\bigl({\rm Sel}_{0,\emptyset}(K,T_{\psi_E})\bigr)=1\quad\Longleftrightarrow\quad\kappa_p\neq 0.
\end{equation}
The first claim in the Theorem now follows from \eqref{eq:str=1} and Lemma\,\ref{lem:-1}, noting that by the work of Rubin \cite{rubin-IMC} proving the Iwasawa main conjecture for $K$, the vanishing of $L(E,1)$ implies the non-triviality of ${\rm Sel}(\Q,V_pE)$. Similarly, the last claim in the Theorem is a direct consequence of \eqref{eq:str=1} and Lemma\,\ref{lem:-1}.  
\end{proof}


\begin{rem}\label{eq:res-0}
The result of Theorem~\ref{thm:B} confirms expectations suggested by the conjectures of Darmon--Rotger \cite{DR2.5} (see esp. Conjecture~3.12 in \emph{op.\,cit.} as specialised to the ``rank $(2,0)$ setting'' in $\S{4.5.3}$), and further shows that the nonvanishing of the restriction map ${\rm res}_p$ is \emph{necessary} for the implication ${\rm dim}_{\Q_p}{\rm Sel}(\bQ,V_pE)=2\Longrightarrow\kappa_p\neq 0$ to hold. 
\end{rem}

\bibliographystyle{amsalpha}
\bibliography{Kato-Schoen-refs}

\end{document}